\UseRawInputEncoding
\documentclass{amsart}

\usepackage{mathrsfs,graphicx,latexsym,tikz,color,euscript}
\usepackage{amsfonts,amssymb,amsmath,amscd,amsthm}
\usepackage{epstopdf}
\usepackage{hyperref}
\usepackage{CJKutf8}

\usepackage{color}

\newcommand{\beq}{\begin{eqnarray*}}
\newcommand{\eeq}{\end{eqnarray*}}

\theoremstyle{plain}
\newtheorem{Thm}{Theorem}[section]
\newtheorem{Prop}[Thm]{Proposition}
\newtheorem{Lem}[Thm]{Lemma}
\newtheorem{Cor}[Thm]{Corollary}

\theoremstyle{definition}

\newtheorem{Rmk}[Thm]{Remark}

\newtheorem*{Acknow}{Acknowledgements}
\newtheorem*{Conflict}{Declaration of Interest Statement}
\newtheorem*{Dataava}{Data Availability Statement}

\newcommand{\authorfootnotes}{\renewcommand\thefootnote{\@fnsymbol\c@footnote}}%

\title[On the S-S estimates for the restricted one-center-two-body problem]
{On the Sundman-Sperling estimates for the restricted one-center-two-body problem}

\begin{document}






\thanks{2020 \em Mathematics Subject Classification\em. 70F07, 70F15, 70F16, 70G75, 70M20. \\
\indent E-mail: s9921803@m99.nthu.edu.tw,\ liulei30@email.sdu.edu.cn. }


\maketitle

\begin{center}

\normalsize
\authorfootnotes
Ku-Jung Hsu\textsuperscript{1} and
Lei Liu\textsuperscript{2}  \par
\bigskip
\textsuperscript{1} School of Mathematical Sciences, Huaqiao University \par

\textsuperscript{2}School of Mathematics, Shandong University

\end{center}


\begin{abstract}
In the past two decades, since the discovery of the figure-8 orbit by  Chenciner and Montgomery,
the variational method has became one of the most popular tools for constructing new solutions of the $N$-body problem and its extended problems.
However, finding solutions to the restricted three-body problem, in particular, the two primaries form a collision Kepler system, remains a great difficulty. One of the major reasons is the essential differences between two-body collisions and three-body collisions.

In this paper, we consider a similar three-body system with less difficulty, i.e. the restricted one-center-two-body system, that is involving a massless particle and a collision Kepler system with one body fixed. It is an intermediate system between the restricted three-body problem and the two-center problem. By an in-depth analysis of the asymptotic behavior of the minimizer, and an argument of critical and infliction points, we prove the Sundman-Sperling estimates near the three-body collision for the minimizers. With these estimates, we provide a class of collision-free solutions with prescribed boundary angles. Finally, under the extended collision Kepler system from Gordon, we constructed a family of periodic and quasi-periodic solutions.
\end{abstract}

\tableofcontents

\section{Introduction}\label{sec:int}

In 2000, Chenciner and Montgomery \cite{CM00} showed the existence of a remarkable periodic solution (called figure-eight solution) of three-body problem. Since then, inspired by this work, various new solutions of the $N$-body problems and $N$-center problems are constructed by variational methods, see \cite{Chen08,ChenYu18,FT04,FGN11,ST13,Yu16} and the references therein. The most crucial step to prove the existence of solutions for $N$-body problems through variational methods is to exclude the collisions of the minimizers.

The currently known methods are mainly the level estimates and the local deformation method. The former is to estimate the minimal action of all the collision paths, then find a test path with action lower than the previous minimal action. The latter is to locally perturb the collision paths near the colliding moments, such that their actions strictly decrease.
It is well-known that, the local deformation method is based on the asymptotic estimates of the paths near their collisions.

The asymptotic estimates of multi-body collisions have been studied by Sundman \cite{Sundman} since 1913, who provided estimates of the moment of inertia for the collision clusters. Another analogous estimates for two-body collisions was proved by Sperling \cite{Sperling69,Sperling70} in 1969, independently. Venturelli \cite{Venturelli02}, Ferrario and Terracini \cite{FT04} provided a general criterion for the Sundman estimates, which also fits the two-body collision. Unfortunately, in restricted multi-body problem, acquiring asymptotic behaviors of multi-body collisions might create more technical difficulties. To our best knowledge, there are no such estimates of multi-body collisions in the restricted multi-body problems. This might be the main reason that only few results regard the restricted multi-body problems by using variational methods, see \cite{KS22,Moeckel05,Shibayama19}.

The simplest restricted $N$-body problem is the restricted one-center-one-body problem (or the Kepler problem). It involves a fixed particle $c$ (called a \em center\em) at the origin with mass $\mu$ and a moving particle $q\in \mathbb{C}$  (called a \em primary\em) with mass $m$. The motion of $q$ is subjected to Newton's universal gravitational law:
\begin{align}\label{eqn:newton}
\ddot{q}=- \frac{\mu q}{\lvert q \rvert ^3}.
\end{align}
The solutions of \eqref{eqn:newton} are either conics or straight lines, and the latter are the only solutions with collision. According to the results in \cite{Sundman}, every solution of \eqref{eqn:newton} satisfies the Sundman-Sperling estimates near their collision. We refer some well-known applications of the Sundman-Sperling estimates to \cite{Chenciner98,FT04,KS22,TV2007,Venturelli02,Wintner41} and the references theirin. 

However, the restricted three-body problems include huge complexities, especially in the analysis of asymptotic behavior near the three-body collisions. There are several difficulties for this. Firstly, it is unclear that the massless particle satisfies Sundman-Sperling estimates. Secondly, the solutions might spin infinitely or with an oscillation near the three-body collisions. Thirdly, there are two singularities that close to one another, and it remains unclear whether such singularities in a restricted multi-body system can be regularized.

In this paper, to reduce the difficulties, we consider the Sundman-Sperling estimates of the three-body collisions in a simplified restricted three-body problems, which involves a collision Kepler system $(q,c)$ and a massless particle $z\in\mathbb{C}$. The motion of $z$ is governed by the following equation:
\begin{align}\label{eqn:1+1+1-body}
\ddot{z}=-\frac{\mu z}{\lvert z\rvert ^3}-\frac{m (z-q(t))}{\lvert z-q(t) \rvert ^3}=\frac{\partial}{\partial z}U(z,t),
\end{align}
where $U(z,t)$ is the time-dependent potential defined by
\begin{align}\label{eqn:potential}
U(z,t)= \frac{\mu}{\lvert z \rvert }+\frac{m}{\lvert z-q(t) \rvert }.
\end{align}
In fact, this is an intermediate problem between the restricted three-body problem and the Euler's problem with two fixed centers.
Since it involves the interaction of one center, one primary and one massless particle, we refer to it as \em the restricted one-center-two-body problem\em.

According to the fact (see \cite{Gordon77}) that the collision moments of the collision Kepler system $(q,c)$ are isolated, without loss of generality, for some $T>0$, we set
\begin{itemize}
\item[$(Q1)$]  $q$ collides with $c$ at moment $0$, i.e. $q(0)=0$.
\item[$(Q2)$]  $q$ doesn't collide with $c$ on $[-T,0)\cup (0,T]$, i.e. $q(t) \neq 0$ on $[-T,0)\cup (0,T]$.
\item[$(Q3)$]  $q$ lies on the negative real axis on $[-T,T]$, i.e. $q|_{[-T,T]} \subset \mathbb R^-:= (-\infty,0]$.
\end{itemize}
In this setting, we allow $|\dot q|$ to be nonzero at moment $\pm T$ so that the energy of $q$ could be even zero or positive. Moreover, one can see that the system (\ref{eqn:newton}) with conditions $(Q1) - (Q3)$ is symmetric with respect to
\begin{align}\label{eqn:symmetric}
q(t)=q(-t), \quad   \text{ for all } \  [-T,T],
\end{align}
and \eqref{eqn:1+1+1-body} is symmetric with respect to the complex conjugation. Therefore, $z(t)$, $\bar z(t)$, $z(-t)$ and $\bar z(-t)$ solve the equation \eqref{eqn:1+1+1-body} at the same time. In fact, for any $a<b$, the equation (\ref{eqn:1+1+1-body}) is the Euler-Lagrange equation of the action functional
\begin{align}\label{eqn:action0}
\mathcal A_{a,b}(z)=\int_{a}^{b}\frac{1}{2}|\dot z|^2+U(z,t)dt.
\end{align}

Consider the path space
$\Omega^{a,b}_{A,B}=\{x\subset H^1([a,b],\mathbb{C}):x(a)\in A,x(b)\in B\},$ where $A,B\subset \mathbb{C}^+=\{x+iy\in \mathbb{C}:y\geq 0\}$ are closed disjoint subsets.
If the minimizer $z(t)$ of $\mathcal{A}_{a,b}$ exists on $\Omega^{a,b}_{A,B}$, then
$z(t)$ is a weak solution of the restricted one-center-two-body problem (\ref{eqn:1+1+1-body}) for $t \in [a,b]$. We note that the minimizer $z(t)$ becomes a classical solution if it does not include any collision. Now, we introduce our main theorem as follows.

\begin{Thm}\label{main1}
Given $T>0$, a collision Kepler system $(q,c)$ satisfying \eqref{eqn:newton} and $(Q1)-(Q3)$, and an action functional $\mathcal{A}_{-T,0}$ $(\mathcal{A}_{0,T})$ as in \eqref{eqn:action0}. Assume $z(t)=r(t)e^{\theta(t)i}$ is a minimizer of $\mathcal{A}_{-T,0}$ $(\mathcal{A}_{0,T})$ on $\Omega^{-T,0}_{A,A_0}$ $(\Omega^{0,T}_{A_0,A})$ with $0\in A_0$. If $z$ admits a three-body collision, i.e. $z(0)=0$, then
\begin{itemize}
\item[$(a)$] The argument $\theta(t)$ is either a constant function $\theta(t)\equiv \{0,\pi\}$ or a  strictly decreasing function on $(-T,0)\ ($strictly increasing function on $(0,T))$. Moreover, the limit angles $\theta_*^{\pm}:= \lim_{t\rightarrow0^{\pm}}\theta(t)$ exist and $\theta_*^{\pm}\in\{0,\pi\}$.
\item[$(b)$] As $t \rightarrow 0^{\pm}$, there exists an $\alpha_*>0$ such that
\begin{equation}\label{eqn:sundman1}
\begin{aligned}
 r(t) &=  \ \alpha_* \lvert q(t) \rvert  +o( \lvert t \rvert ^{\frac{2}{3}})= \ \alpha_*\lambda_{\mu} \lvert t \rvert ^{\frac{2}{3}}+o( \lvert t \rvert ^{\frac{2}{3}}), \\
\dot r(t)  &= \ \alpha_*\frac{d}{dt} \lvert q(t) \rvert  +o( \lvert t \rvert ^{-\frac{1}{3}})= \pm\frac{2}{3}\alpha_*\lambda_{\mu} \lvert t \rvert ^{-\frac{1}{3}}+o( \lvert t \rvert ^{-\frac{1}{3}}), \\
\ddot r(t)  & = \ \alpha_*\frac{d^2}{dt^2} \lvert q(t) \rvert  +o( \lvert t \rvert ^{-\frac{4}{3}})=-\frac{2}{9}\alpha_*\lambda_{\mu} \lvert t \rvert ^{-\frac{4}{3}}+o( \lvert t \rvert ^{-\frac{4}{3}}),
\end{aligned}
\end{equation}
where $\lambda_{\mu}=(9\mu/2)^{1/3}$.
\begin{itemize}
\item[$(b_1)$] If $\theta_*^-=0$ ($\theta_*^+=0$), then $\alpha_*=\alpha_2>1$, which is the unique solution of
\begin{equation}\label{eqn:h(a,0)}
a^3-1-\frac{m}{\mu}\frac{a^2}{(a+1)^2}=0\quad \text{on}\quad a\in[0,+\infty).
\end{equation}

\item[$(b_2)$] If $\theta_*^-=\pi$ ($\theta_*^+=\pi$), then $\alpha_*=\alpha_1\in(0,1)$ or $\alpha_*=\alpha_3\in(1,+\infty)$, where $\alpha_1,\alpha_3$ are the unique two solutions of
\begin{equation}\label{eqn:h(a,pi)}
a^3-1-\frac{m}{\mu}\frac{a^2}{(a-1)^2}=0\quad \text{on}\quad a\in[0,+\infty).
\end{equation}
Moreover, the former case occurs if $z(t)\in(q(t),0)$ and the latter case occurs if $z(t)\in(-\infty,q(t))$, for any $t\in[-T,0)\ (t\in(0,T])$.
\end{itemize}
\end{itemize}
\end{Thm}

This theorem mainly characterizes the asymptotic behaviours of the minimizer $z(t)$ near the three-body collision. Term $(a)$ describes the behaviour of the argument function $\theta(t)$. Term $(b)$ describes the behaviours of the norm $r(t)$, which is so-called the Sundman-Sperling estimates for the minimizer $z(t)$. The main technique of the proof is the analysis of critical and inflection points together with the properties of the minimizer.

Generally, for any restricted multi-body system, we believe that the Theorem \ref{main1} remains valid, provided the collision involves one center, one primary, and one massless particle. This is because when the massless particle approaches the three-body collision,  the effect of the other non-colliding particles is negligible.

Notice that, the montonicity of $\theta(t)$ for the minimizer $z(t)$ is the most fundamental property in the restricted one-center-two-body problem. It highly relies on the monotonicity of the potential function $U(z,t)$, i.e. the potential $U$ is strictly decreasing from $\theta=\pi$ to $\theta=0$ for any fixed $r>0$. However, the potential function possesses no such monotonicity in the restricted three-body problem, which leads to great difficulties in the analysis. 



As an application, we consider the existence of collision-free solutions in the restricted one-center-two-body problem for all precribed boundary angles $(\phi_1,\phi_2)\in[0,\pi]\times[0,\pi]$ with $\phi_1\neq \phi_2$, which is a solution $z(t)$ jointing from the ray $e^{\phi_1i}\mathbb{R}^+$ to $e^{\phi_2i}\mathbb{R}^+$ in $t\in[-T,0]$ ($t\in[0,T]$). In the one center problem, it is well-known that the collision-free solutions exist for all $(\phi_1,\phi_2)$ with $|\phi_1-\phi_2|\in(0,\pi)$, see \cite{Chen03}, \cite[Prop.3]{Chen08}. In this paper, by using the natures of the minimizers and Theorem \ref{main1}, we obtain the following results.


\begin{Thm}\label{thm:main1.1}
Given $T>0$  and a collision Kepler system $(q,c)$ which satisfies (\ref{eqn:newton}) and  $(Q1) - (Q3)$.
For any $(\phi,\phi_0) \in [0,\pi)\times [0,\pi/2]$ with $\phi \neq \phi_0$, the restricted one-center-two-body problem (\ref{eqn:1+1+1-body}) with system $(q,c)$ possesses a solution $z(t)=r(t)e^{\theta(t) i}$ satisfying the following properties:
\begin{itemize}
\item[$(a)$] $z$ is collision-free on $[-T,0]$.
\item[$(b)$] $\theta(-T)=\phi$ and  $\theta(0) = \phi_0$.
\item[$(c)$] There is a unique $t_* \in [-T,0]$ such that  $\theta(t_*) \in [0,\min\{\phi,\phi_0\}]$ and $\theta(t)$ is strictly decreasing on $[-T,t_*]$ and strictly increasing on $[t_*,0]$. Especially, if $\min\{\phi,\phi_0\}=0$, $\theta(t)$ is strictly monotone on $[-T,0]$.
\item[$(d)$] $\dot z$ is orthogonal to the ray $e^{\phi i}\mathbb R^+$ at $z(-T)$ and orthogonal to the ray $e^{\phi_0 i}\mathbb R^+$ at $z(0)$.
\end{itemize}
\end{Thm}

According to the symmetry (\ref{eqn:symmetric}), a similar result as Theorem \ref{thm:main1.1} holds after the collision moment $t=0$ between $q$ and $c$.

\begin{Thm}\label{thm:main1.2}
Given $T>0$  and a collision Kepler system $(q,c)$ which satisfies (\ref{eqn:newton}) and  $(Q1) - (Q3)$. For any $(\phi,\phi_0) \in [0,\pi)\times [0,\pi/2]$ with $\phi \neq \phi_0$, the restricted one-center-two-body problem (\ref{eqn:1+1+1-body}) with system $(q,c)$ possesses a solution $z(t)=r(t)e^{\theta(t) i}$ satisfying the following properties:
\begin{itemize}
\item[$(a)$] $z$ is collision-free on $[0,T]$.
\item[$(b)$] $\theta(0)=\phi_0$ and  $\theta(T) = \phi$.
\item[$(c)$]  there is a unique $t_* \in [0,T]$ such that  $\theta(t_*) \in [0,\min\{\phi,\phi_0\}]$ and $\theta(t)$ is strictly decreasing on $[0,t_*]$ and strictly  increasing on $[t_*,T]$. In particular, if $\min\{\phi,\phi_0\}=0$, then $\theta(t)$ is strictly monotone on $[0,T]$.
\item[$(d)$] $\dot z$ is orthogonal to $e^{\phi_0 i}\mathbb R^+$ at $z(0)$ and orthogonal to $e^{\phi i}\mathbb R^+$ at $z(T)$.
\end{itemize}
\end{Thm}
As a conclusion, we have the following corollaries.
\begin{Rmk}\label{thm:main1.3}
When $\phi_0=0$, it follows from Theorem \ref{thm:main1.1}, \ref{thm:main1.2} that the argument $\theta(t)$ of the collision-free solution $z(t)$ is strictly decreasing on $[-T,0]$ and strictly increasing on $[0,T]$.
\end{Rmk}

Theorem~\ref{thm:main1.1}, \ref{thm:main1.2} show
the existence of the collision-free solution of $(\ref{eqn:1+1+1-body})$ for any choice of masses $\mu,m>0$ and boundary angles $(\phi,\phi_0)\in[0,\pi)\times [0,\pi/2]$ with $\phi\neq \phi_0$.
In the classical restricted three-body problem, different energy on the two primaries will alter the nature of the problem, significantly.
It is worth noting that Theorem \ref{thm:main1.1}, \ref{thm:main1.2} are independent of the choice of energy on the two-body system $(q,c)$.

Notice that, in restricted multi-body problem,  almost all two-body collisions of minimizers can be excluded by local deformation. However, there is also a lack of asymptotic estimates for the three-body collisions, such as the Sundman-Sperling estimates. This makes the exclusion of the three-body collisions much more challenging.

In Theorem \ref{thm:main1.1}, \ref{thm:main1.2}, as an application of Theorem \ref{main1}, we successfully exclude the three-body collision for the minimizer in the restricted one-center-two-body problem. Unfortunately, there is no regularization to the three-body collisions in our problem, unlike the Levi-Civita regularization in the two-body collisions. This causes that the action of the local deformation paths are highly difficult to estimate under the behavior of the two singularities, and then the boundary angle $\phi_0$ can only be choosed in $[0,\pi/2]$ rather than $[0,\pi]$. More specifically, the regularization method requests more regularity than we have.

Based on our results above, although the restricted one-center-two-body problem is not the classical restricted three-body problem, the authors believe that the methods in this paper and the extension of Sundman-Sperling estimates for three-body collisions will be useful in advancing the study of collisions in celestial mechanics, and provides a promising direction for future investigations into general three-body collisions or even multi-body collisions.

\begin{Rmk}
For the sake of intuition, several numerical examples of the solution $z(t)$ are listed in Figure~\ref{fig:mainthm} including different masses $(\mu,m)$ and boundary angles $(\phi, \phi_0)$, in which the collision Kepler system $(q,c)$ satisfies the boundary conditions $q(\pm T)=-1$ and $\dot{q}(\pm T)=0$.
In Figure~\ref{fig:mainthm_angle}, we also provide an example for Theorem~\ref{thm:main1.1}(c), \ref{thm:main1.2}(c), where the argument $\theta(t)$ of the solution is not monotonic.
\end{Rmk}

\begin{figure}[ht]
\begin{center}
\includegraphics[width=0.7 \textwidth]{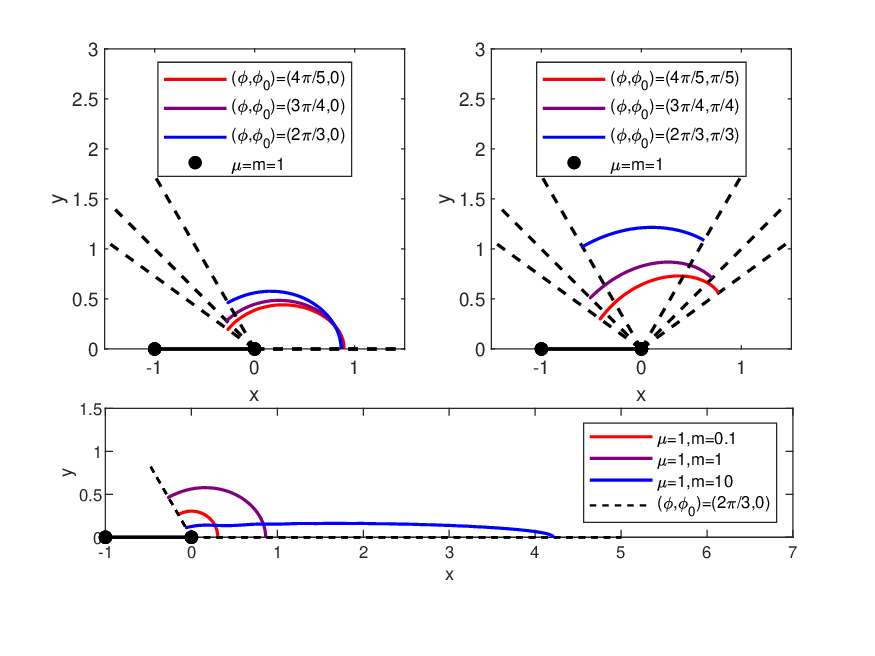}
\vspace{-7mm}
\caption{The solutions in Theorem~\ref{thm:main1.1}, \ref{thm:main1.2} with different choice of masses and boundary angles.}
\label{fig:mainthm}
\includegraphics[width=0.7 \textwidth]{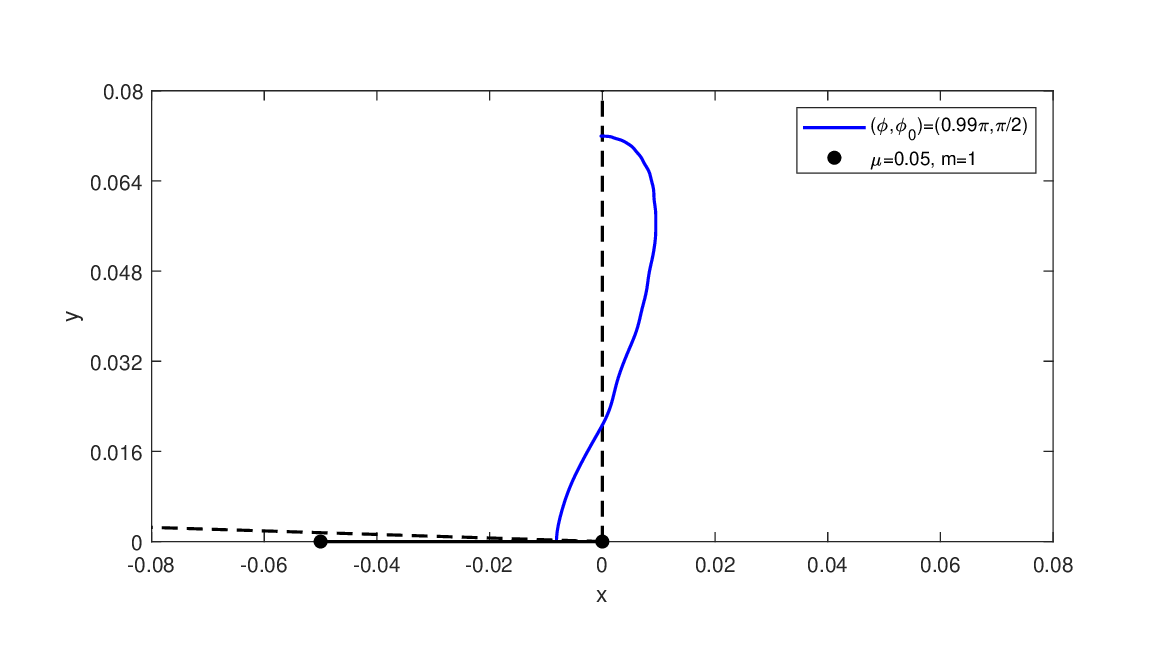}
\vspace{-3mm}
\caption{An example of Theorem~\ref{thm:main1.1}(c), in which $q$
 satisfies $(q(-T),\dot q(-T))=(-0.05,1)$ and $q$ is much heavier than $c$.}
\label{fig:mainthm_angle}
\end{center}
\end{figure}
This paper is organized as follows. In Section~\ref{sec2}, we recall the Sundman-Sperling estimates for the two-body collisions, and introduce a classical approach for the exclusion of the two-body collisions. In Section~\ref{sec21}, we show some important monotonicities for both potential function $U$ and action minimizers of the restricted one-center-two-body problem. In Section \ref{sec3}, by analysing the asymptotic behaviors for the three-body collision, we prove the Sundman-Sperling estimates for the action minimizers (Theorem \ref{main1}). In Section \ref{sec4}, as an application of Theorem \ref{main1}, we prove the existence of the collision-free solutions with prescribed boundary angles in the restricted one-center-two-body problem (Theorem \ref{thm:main1.1}, \ref{thm:main1.2}). Finally, in Appendix \ref{sec:periodic}, we construct a class of periodic and quasi-periodic orbits under the extended collision Kepler system.


\section{Preliminaries}\label{sec2}



\subsection{Asymptotic behavior near two-body collision}\label{subsec: 2.1}
This subsection will review the asymptotic analysis near a two-body collision. Given a collision Kepler system $(q,c)$ which satisfies (\ref{eqn:newton}) and $(Q1) - (Q3)$, and let $z$ be a solution of the restricted one-center-two-body problem (\ref{eqn:1+1+1-body}). In this problem, there are three possibilities of two-body collisions: collisions between $q$ and $c$, $z$ and $c$, and  $z$ and $q$. The collision moments between $q$ and $c$ are clearly isolated, and the following lemma shows that the collision moments are isolated for the other two types of two-body collisions. The proof is based on the regularization method, which we refer to \cite[Sec.3.3]{Chenciner02}, \cite[Sec.4.1]{Venturelli02} and the proof is omitted here.

\begin{Lem}\label{lem:isolated}
The sets of two-body collision moments $\triangle_c(q)=\{t\in\mathbb{R}: q(t)=0\ \text{and}\ z(t)\neq 0\}$, $\triangle_{c}(z):=\{t\in \mathbb R :  z(t)=c \text{ and } z(t) \neq q(t) \}$ and $\triangle_{q}(z):=\{t\in \mathbb R :  z(t)=q(t) \text{ and } z(t) \neq c \}$ are isolated.
\end{Lem}


Next, we introduce a famous work of Sperling \cite{Sperling69}, in which the asymptotic behaviors of a particle near two-body collisions are given in the perturbed Kepler problem. As an application, we can obtain the asymptotic behaviors of the aforementioned three types of two-body collisions in the restricted one-center-two-body problem.

\begin{Prop}\cite{Sperling69} \label{pro:Sperling69}
Consider the perturbed Kepler problem
\begin{align}\label{eqn:sperling}
\ddot{u}=-\hat{\mu} \frac{u}{ \lvert u \rvert ^{3}}+P(u,t),
\end{align}
where $\hat{\mu}>0$ and $P(u,t)$ is a bounded and continuous function near $0$. Assume the solution $u$ of (\ref{eqn:sperling}) has a collision at $\tau$, that is $u(\tau)=0$, then there exist $\lambda=(9\hat{\mu}/2)^{1/3} $ and  $\eta^{\pm} \in \mathbb{S} :=\{x \in \mathbb C: \  \lvert x \rvert =1 \}$ such that
\begin{enumerate}
\item[$(a)$] $ \lvert u(t) \rvert = \lambda \lvert t-\tau \rvert ^{\frac{2}{3}}+o( \lvert t-\tau \rvert ^{\frac{2}{3}})$ as $t \rightarrow \tau^\pm$.\vspace{1ex}
\item[$(b)$] $\frac{d}{dt} \lvert u(t) \rvert = \pm\frac{2}{3}\lambda \lvert t-\tau \rvert ^{-\frac{1}{3}}+o( \lvert t-\tau \rvert ^{-\frac{1}{3}})$ as $t \rightarrow \tau^\pm$.\vspace{1ex}
\item[$(c)$] $\frac{d^2}{dt^2} \lvert u(t) \rvert = -\frac{2}{9}\lambda \lvert t-\tau \rvert ^{-\frac{4}{3}}+o( \lvert t-\tau \rvert ^{-\frac{4}{3}})$ as $t \rightarrow \tau^\pm$.\vspace{1ex}
\item[$(d)$] $\lim_{t \rightarrow \tau^\pm}\frac{u(t)}{ \lvert u(t) \rvert }=\eta^{\pm}$ exist.
\end{enumerate}
\end{Prop}

The properties $(a)$ - $(c)$ demonstrate that the asymptotic positions, velocities and accelerations of colliding particles satisfy the Sundman-Sperling estimates (see \cite{Sperling69,Sperling70,Sundman}); the property $(d)$ further reveals the existence of the asymptotic angles for colliding particles.

In system $(q,c)$, the motions of $q$, $z$ and $z-q$ satisfy the equations $(\ref{eqn:newton})$, $(\ref{eqn:1+1+1-body})$ and
\begin{align*}
\frac{d^2}{dt^2}(z-q)= -\frac{m (z-q)}{ \lvert z-q \rvert ^3} + \frac{\mu q}{ \lvert q \rvert ^3}-\frac{\mu z}{ \lvert z \rvert ^3},
\end{align*}
respectively. We can also rewrite these equations as the perturbed Kepler systems
$(\ref{eqn:sperling})$:
\begin{align*}
\ddot{q} = - \frac{\mu q}{ \lvert q \rvert ^3}+P_0(q,t),  \quad\quad\quad\quad \ \ &\text{where} \ \ P_0(u,t) \equiv 0, \\
\ddot{z} = -\frac{\mu z}{ \lvert z \rvert ^3}+P_1(z,t),  \quad\quad\quad\quad\ \ &\text{where} \ \  P_1(u,t) = -\frac{m (u-q(t))}{ \lvert u-q(t) \rvert ^3},\\
\frac{d^2}{dt^2}(z-q) = -\frac{m (z-q)}{ \lvert z-q \rvert ^3} +P_2(z-q,t), \ \ &\text{where} \ \  P_2(u,t) = -\frac{\mu (u+q(t))}{ \lvert u+q(t) \rvert ^3} + \frac{\mu q(t)}{ \lvert q(t) \rvert ^3}.
\end{align*}

Since $P_0,P_1,P_2$ are bounded and continuous near the collision moment sets $\triangle_{c}(q)$, $\triangle_{c}(z)$ and $\triangle_{q}(z)$, as an application of Proposition~\ref{pro:Sperling69}, we obtain the Sundman-Sperling estimates for $q,z$ and $z-q$ as follows.

\begin{Prop}\label{pro:sundmanq-}
There exists $\lambda_{\mu} =(9\mu /2)^{1/3}$ such that
\begin{enumerate}
\item[$(a)$] $ \lvert q(t) \rvert = \lambda_{\mu} \lvert t \rvert ^{2/3}+o( \lvert t \rvert ^{2/3})$ \ as $t \rightarrow 0^\pm$.
\item[$(b)$] $\frac{d}{dt} \lvert q(t) \rvert = \pm\frac{2}{3}\lambda_{\mu} \lvert t \rvert ^{-1/3}+o( \lvert t \rvert ^{-1/3})$ \ as $t \rightarrow 0^\pm$.
\item[$(c)$] $\frac{d^2}{dt^2} \lvert q(t) \rvert = -\frac{2}{9}\lambda_{\mu} \lvert t \rvert ^{-4/3}+o( \lvert t \rvert ^{-4/3})$ \ as $t \rightarrow 0^\pm$.
\end{enumerate}
\end{Prop}

\begin{Prop}\label{pro:sundmanzc}
Assume $\tau \in \triangle_{c}(z)$,
there exist $\lambda_{\mu} =(9\mu /2)^{1/3}$ and $\theta^{\pm}_{c,\tau} \in \mathbb R$ such that
\begin{enumerate}
\item[$(a)$] $ \lvert z(t) \rvert = \lambda_{\mu} \lvert t-\tau \rvert ^{2/3}+o( \lvert t-\tau \rvert ^{2/3})$ \ as $t \rightarrow \tau^\pm$.
\item[$(b)$] $\frac{d}{dt} \lvert z(t) \rvert = \pm\frac{2}{3}\lambda_{\mu} \lvert t-\tau \rvert ^{-1/3}+o( \lvert t-\tau \rvert ^{-1/3})$ \ as $t \rightarrow \tau^\pm$.
\item[$(c)$] $\frac{d^2}{dt^2} \lvert z(t) \rvert = -\frac{2}{9}\lambda_{\mu} \lvert t-\tau \rvert ^{-4/3}+o( \lvert t-\tau \rvert ^{-4/3})$ \ as $t \rightarrow \tau^\pm$.
\item[$(d)$] $\lim_{t \rightarrow \tau^\pm}\frac{z(t)}{ \lvert z(t) \rvert }=e^{i\theta^{\pm}_{c,\tau}}$ exist.
\end{enumerate}
\end{Prop}

\begin{Prop}\label{pro:sundmanzq}
Assume $\tau \in \triangle_{q}(z)$,
there exist $\lambda_{m} =(9m /2)^{1/3}$ and $\theta^{\pm}_{q,\tau}\in \mathbb R$ such that
\begin{enumerate}
\item[$(a)$] $ \lvert z(t)-q(t) \rvert = \lambda_{m}  \lvert t-\tau \rvert ^{2/3}+o( \lvert t-\tau \rvert ^{2/3})$ \ as $t \rightarrow \tau^\pm$.
\item[$(b)$] $\frac{d}{dt} \lvert z(t)-q(t) \rvert = \pm\frac{2}{3}\lambda_{m}  \lvert t-\tau \rvert ^{-1/3}+o( \lvert t-\tau \rvert ^{-1/3})$ \ as $t \rightarrow \tau^\pm$.
\item[$(c)$] $\frac{d^2}{dt^2} \lvert z(t)-q(t) \rvert = -\frac{2}{9}\lambda_{m}  \lvert t-\tau \rvert ^{-4/3}+o( \lvert t-\tau \rvert ^{-4/3})$ \ as $t \rightarrow \tau^\pm$.
\item[$(d)$] $\lim_{t \rightarrow \tau^\pm}\frac{z(t)-q(t)}{ \lvert z(t)-q(t) \rvert }=e^{i\theta^{\pm}_{q,\tau}}$ exist.
\end{enumerate}
\end{Prop}

\subsection{Exclusion of two-body collisions for minimizers}\label{subsec: exclude twwo-body collision}

In this subsection, we follow the local deformation method,  which is one of the main approaches to excluding two-body collisions for minimizers, see \cite{ChenYu18,ST13,Yu16}. More precisely, under the blow-up technique, we can construct a new path without any two-body collision, and its action is strictly below the original colliding path, so that the minimizer does not involve any two-body collision.

We first list some useful results. For details of the proofs, we refer to \cite{FGN11,Hsu22,ST12,Yu16}.

Given an action functional $\mathcal{A}_{a,b}$ on $\Omega^{a,b}_{A,B}$ as in \eqref{eqn:action0}. Consider a colliding path $z(t)$ of the restricted one-center-two-body problem \eqref{eqn:1+1+1-body} with $\Delta_{\xi}(z)\cap (\tau-\delta,\tau+\delta)=\{\tau\}$, $\xi\in\{c,q\}$, for some $\delta>0$.

\begin{Prop}\label{pro:local1}
If $ \lvert \theta_{\xi,\tau}^+-\theta_{\xi,\tau}^- \rvert <2\pi$ where $\theta_{\xi,\tau}^{\pm}$ is given by Proposition~\ref{pro:sundmanzc}, \ref{pro:sundmanzq}, then for any $\delta^*>0$ sufficiently small,
there exists an $\epsilon=\epsilon(\delta^*)>0$  with $\lim_{\delta^* \rightarrow 0}\epsilon(\delta^*)=0$
and a collision-free path $\eta \in H^1([\tau-\delta,\tau+\delta],\mathbb C)$ such that
\begin{enumerate}
\item[$(a)$] $\eta(t)=z(t)$ for any $t \in [\tau-\delta,\tau+\delta]\backslash[\tau-\delta^*,\tau+\delta^*]$.
\item[$(b)$] $ \lvert \eta(t)-\xi(t) \rvert  \leq \epsilon$  for any $t \in [\tau-\delta^*,\tau+\delta^*]$.
\item[$(c)$] $Arg(\eta(\tau\pm\delta^*)-\xi(\tau\pm\delta^*))=Arg(z(\tau\pm\delta^*)-\xi(\tau\pm\delta^*))$.
\item[$(d)$] $\mathcal A_{\tau-\delta,\tau+\delta}(\eta)<\mathcal A_{\tau-\delta,\tau+\delta}(z)$.
\end{enumerate}
\end{Prop}

For a colliding path $z(t)$ of \eqref{eqn:1+1+1-body} with $\Delta_{\xi}(z)\cap [\tau,\tau+\delta)=\{\tau\}$ or $\Delta_{\xi}(z)\cap (\tau-\delta,\tau]=\{\tau\}$, $\xi\in\{c,q\}$, for some $\delta>0$, we have the following results.


\begin{Prop}\label{pro:local2}
Given $\theta_0 \in \mathbb R$. If $ \lvert \theta_{\xi,\tau}^+-\theta_0 \rvert <\pi$, where $\theta_{\xi,\tau}^{+}$ is given by Proposition~\ref{pro:sundmanzc}, \ref{pro:sundmanzq}, then for any $\delta^*>0$ sufficiently small,
there exists an $\epsilon=\epsilon(\delta^*)>0$  with $\lim_{\delta^* \rightarrow 0}\epsilon(\delta^*)=0$
and a collision-free path $\eta \in H^1([\tau,\tau+\delta],\mathbb C)$ such that
\begin{enumerate}
\item[$(a)$] $\eta(t)=z(t)$ for any $t \in [\tau+\delta^*,\tau+\delta]$.
\item[$(b)$] $ \lvert \eta(t)-\xi(t) \rvert  \leq \epsilon$  for any $t \in [\tau,\tau+\delta^*]$.
\item[$(c)$] $Arg(\eta(\tau)-\xi(\tau))=\theta_0$ and $Arg(\eta(\tau+\delta^*)-\xi(\tau+\delta^*))=Arg(z(\tau+\delta^*)-\xi(\tau+\delta^*))$.
\item[$(d)$] $\mathcal A_{\tau,\tau+\delta}(\eta)<\mathcal A_{\tau,\tau+\delta}(z)$.
\end{enumerate}
\end{Prop}

\begin{Prop}\label{pro:local3}
Given $\theta_0 \in \mathbb R$. If $ \lvert \theta_{\xi,\tau}^--\theta_0 \rvert <\pi$, where $\theta_{\xi,\tau}^{-}$ is given by Proposition~\ref{pro:sundmanzc}, \ref{pro:sundmanzq}, then for any $\delta^*>0$ sufficiently small, there exists an $\epsilon=\epsilon(\delta^*)>0$  with $\lim_{\delta^* \rightarrow 0}\epsilon(\delta^*)=0$
and a collision-free path $\eta \in H^1([\tau-\delta,\tau],\mathbb C)$ such that
\begin{enumerate}
\item[$(a)$] $\eta(t)=z(t)$ for any $t \in [\tau-\delta,\tau-\delta^*]$.
\item[$(b)$] $ \lvert \eta(t)-\xi(t) \rvert  \leq \epsilon$  for any $t \in [\tau-\delta^*,\tau]$.
\item[$(c)$] $Arg(\eta(\tau)-\xi(\tau))=\theta_0$ and $Arg(\eta(\tau-\delta^*)-\xi(\tau-\delta^*))=Arg(z(\tau-\delta^*)-\xi(\tau-\delta^*))$.
\item[$(d)$] $\mathcal A_{\tau-\delta,\tau}(\eta)<\mathcal A_{\tau-\delta,\tau}(z)$.
\end{enumerate}
\end{Prop}

As an application, we prove the following result.

\begin{Thm}\label{thm:two-body}
Given $T>0$, a collision Kepler system $(q,c)$ satisfying (\ref{eqn:newton}) and  $(Q1) - (Q3)$. Let $\mathcal{A}_{-T,0}\ (\mathcal{A}_{0,T})$ be an action functional on $\Omega^{-T,0}_{A,A_0}\ (\Omega^{0,T}_{A_0,A})$ as in \eqref{eqn:action0} and assume $z(t)$ is an associated minimizer. Then $z(t)$ possesses no two-body collision on $(-T,0]\ (\text{or}\ [0,T))$.
\end{Thm}

\begin{proof}
According to the reversibility of \eqref{eqn:1+1+1-body}, it sufficient to prove for any $t\in (-T,0]$. Since \eqref{eqn:1+1+1-body} is also symmetric with respect to the real axis, without loss of generality, we assume $z(t)=r(t)e^{i\theta(t)}$ satisfies $\theta(t) \in [0,\pi]$ for all $t\in[-T,0]$.

Suppose $\Delta_{q}(z)\cup \Delta_{c}(z)\neq \emptyset$. Choose a collision moment $\tau \in \Delta_{\xi}(z)$ with $\xi\in\{c,q\}$. We first observe that $z$ is impossible to experience a two-body collision with $\xi$ at $t=0$ since $q=c=0$ at $t=0$.  This tells us that $\tau \neq 0$.

If $\tau \in (-T,0)$, by Proposition~\ref{pro:sundmanzc}, \ref{pro:sundmanzq} and the assumption $\theta \in [0,\pi]$,  the limit angles $\theta^{\pm}_{\xi,\tau} \in [0,\pi]$ exist.
According to Proposition~\ref{pro:local1}, we get a contradiction to the assumption that $z$ is a minimizer.
\end{proof}

\section{Monotonicities in restricted one-center-two-body problem}\label{sec21}
In restricted one-center-two-body problem, we explore several monotonicities for both the potential function $U$ and the action minimizer $z$. Based on these important properties, it is available to prove Sundman-Sperling estimates for the action minimizer of this problem, i.e. Theorem \ref{main1}.

\subsection{Monotonicity of the potential function $U$}
In this subsection, to characterize the restricted one-center-two-body problem \eqref{eqn:1+1+1-body}, we firstly reveal the monotonicity of the potential $U(z,t)$, which is one of the fundamental property in this problem. With this monotonicity, a series of nontrivial conclusions can be proven in the later sections.

\begin{Lem}\label{lem:anglepotential}
Given potential $U$ as in (\ref{eqn:potential}). For any fixed $t\in \mathbb{R}$ and $r\in\mathbb{R}^+$, we have
$$U(re^{i\theta_1},t) < U(re^{i\theta_2},t),\ \forall\ 0\leq |\theta_1| < |\theta_2|\leq \pi.$$
In particular, for any $\eta\in[-\pi,\pi]\backslash\{0\}$,  we have $U(r,t) < U(re^{i\eta},t)$.
\end{Lem}

\begin{proof}
Let $t\in \mathbb R$ and  $u=re^{i\theta}$ with $r \in \mathbb{R}^+$ and $\theta\in[-\pi,0]$ (resp. $[0, \pi]$),
then we can write $U(u,t)$ as
\begin{align*}
U(u,t)=\frac{\mu}{ \lvert u \rvert }+\frac{m}{ \lvert u-q(t) \rvert }= \frac{\mu}{r}+\frac{m}{\left(r^2+ \lvert q(t) \rvert ^2+2r \lvert q(t) \rvert \cos\theta\right)^{1/2}}.
\end{align*}
We see that $U(u,t)$ is strictly increasing (resp.  strictly decreasing) on $[0,\pi]$ (resp.  $[-\pi,0]$) with respect to $\theta$. Hence, the lemma holds.
\end{proof}

\subsection{Monotonicity of the arguments for an action minimizer $z$ and $\dot z$}\label{subsec: 2.2}

Based on the monotonicity of the potential function $U$ in Lemma \ref{lem:anglepotential}, we can further explore the monotonicity of the argument for the action minimizer.

Given $T>0$. Assume $z(t)$ is the action minimizer of $\mathcal{A}_{-T,0}$ \eqref{eqn:action0} on the path space $\Omega^{-T,0}_{A,A_0}\subset H^1([-T,0],\mathbb{C})$. By Lemma \ref{lem:isolated}, there exists a $t_0\in(0,T)$ such that $z(t)$ is smooth on $[-t_0,0)$ (resp. $(0,t_0]$). Next, we consider the following two cases: $z(0)\neq 0$ and $z(0)=0$. For the later case, we need to assume $0\in A_0$. The former means there is no three-body collision at moment $0$. Therefore, we can write $\theta(0)=\lim_{t\rightarrow0^-}\theta(t)$ (resp. $\theta(0)=\lim_{t\rightarrow0^+}\theta(t)$). The latter means there is a three-body collision at moment $0$. Since these two cases have different nature, in this subsection, we will deal with them separately.

Firstly, we assume $z(0)\neq 0$. Since the restricted one-center-two-body problem \eqref{eqn:1+1+1-body} satisfies the symmetry \eqref{eqn:symmetric}, we only need to consider $z(t)$ in $t\in[-t_0,0]$. Moreover,
since the potential $U$ is symmetric with respect to the real axis, without loss of generality, we also assume $\theta(-t_0)\in[0,\pi]$, that is $z(-t_0)\in \mathbb{C}^+_*:=\{x+iy\in\mathbb{C}: y\geq 0\}\setminus\{0\}$. Then we prove the monotonicity of the argument $\theta(t)$ as follows.

\begin{Lem}\label{lem:monotonity0}
Given a $t_0\in(0,T]$ and an action minimizer $z(t)=r(t)e^{\theta(t) i}$ of $\mathcal A_{-T,0}$ on $\Omega^{-T,0}_{A,A_0}$. Assume $z(t)$ is smooth on $(-t_0,0)$, $z(0)\neq 0$ and $z(-t_0)\in \mathbb{C}^+_*$, then
\begin{enumerate}
\item[$(a)$] the argument $\theta(t)$ has no local maximum in $\{t\in(-t_0,0):\theta(t)\in(0,\pi)\}$. In particular, $\ddot \theta(t)\geq0$ if $\dot \theta(t)=0$.

\item[$(b)$] the argument $\theta(t)$ has no local minimum in $\{t\in(-t_0,0):\theta(t)\in(-\pi,0)\}$. In particular, $\ddot \theta(t)\leq 0$ if $\dot \theta(t)=0$.

\item[$(c)$] if there exists a moment $\tilde t\in(-t_0,0)$ such that $\dot \theta(\tilde t)=0$ and $\ddot \theta(\tilde t)=0$, then the argument $\theta(t)\equiv 0$ or  $\theta(t)\equiv\pi$.
\end{enumerate}
\end{Lem}
\begin{proof}
To prove (a). By contradiction, there exists $\hat t,\hat{t}_1, \hat{t}_2\in(-t_0,0)$ and $\epsilon>0$ with $\hat{t}_1 <\hat{t}-\epsilon<\hat{t}<\hat{t}+\epsilon< \hat{t}_2 $ such that  $\theta(\hat{t}_1)=\theta(\hat{t}_2)<\theta(t)$  for $t\in (\hat{t}-\epsilon,\hat{t}+\epsilon)$. Define $\hat{\theta}(t):=\theta(\hat{t}_1)$ on $[\hat{t}_1,\hat{t}_2]$ and  $\hat{\theta}(t):=\theta(t)$ on $[-T,0]\backslash[\hat{t}_1,\hat{t}_2]$, and let $\hat{z}(t)=r(t)e^{i\hat{\theta}(t)}$.
Note that, by  Lemma~\ref{lem:anglepotential},  $U(z,t) \geq U(\hat{z},t)$ and $\dot{\hat{\theta}}(t)=0$ on $(\hat{t}_1,\hat{t}_2)$,  and $U(z,t) > U(\hat{z},t)$ on $(\hat{t}-\epsilon,\hat{t}+\epsilon)$.
Then
\begin{align*}
\mathcal A_{-T,0}(z)-\mathcal A_{-T,0}(\hat{z})
&=\int^{\hat{t}_2}_{\hat{t}_1}r^2\dot{\theta}^2(t)-r^2\dot{\hat{\theta}}^2(t) +U(z,t)-U(\hat{z},t) dt \\
&\geq \int^{\hat{t}+\epsilon}_{\hat{t}-\epsilon}r^2\dot{\theta}^2(t) +U(z,t)-U(\hat{z},t) dt >0.
\end{align*}
This gives a contradiction to the minimizer $z$, (a) holds. Since $\mathcal{A}_{-T,0}$ is invariant under the complex conjugation, we can also prove (b), similarly.

To prove (c). If the moment $\tilde t\in(-t_0,0)$ satisfies $\dot{\theta}(\tilde t)=\ddot{\theta}(\tilde t)=0$, then we have
\begin{align*}
\ddot z(\tilde{t})=\left(\ddot r(\tilde{t})-r(\tilde{t})\dot\theta(\tilde{t})^2\right)e^{i\theta(\tilde{t})}+\left(2\dot r(\tilde{t})\dot\theta(\tilde{t})+r(\tilde{t})\ddot \theta(\tilde{t})\right)e^{i\left(\theta(\tilde{t})+\frac{\pi}{2}\right)}
=\ddot r(\tilde{t})e^{i\theta(\tilde{t})}.
\end{align*}
This implies $\theta(\tilde t)\in\{0,\pi\}$ since the force of $z(\tilde t)$ never points to the original unless $z(\tilde t)\in \mathbb{R}\setminus \{0\}$. By the existence and uniqueness theorem, the minimizer $z$ lies on the real axis on $[-T,0]$. This implies (c).
\end{proof}

\begin{Cor}\label{cor:monotonity1}
Given a $t_0\in(0,T]$ and an action minimizer $z(t)=r(t)e^{\theta(t) i}$ of $\mathcal A_{-T,0}$ on $\Omega^{-T,0}_{A,A_0}$. Assume $z(t)\in\mathbb{C}^+_*$ is smooth on $(-t_0,0)$, then one of the following situations must happen.
\begin{enumerate}
\item[$(a)$] the argument $\theta(t) \equiv 0$ on $[-t_0,0]$.
\item[$(b)$] there is a unique $t_* \in [-t_0,0]$ such that $\theta(t_*)\geq 0$, $\theta$ is strictly decreasing on $[-t_0,t_*]$ and strictly  increasing on $[t_*,0]$.
\item[$(c)$] the argument $\theta(t) \equiv \pi$ on $[-t_0,0]$.
\end{enumerate}
\end{Cor}
\begin{proof}
By assumption, for any $\tilde t\in(-t_0,0)$ with $\dot \theta(\tilde t)=0$, Lemma \ref{lem:monotonity0}(a) implies that $\ddot \theta(\tilde t)\geq 0$. If there is a $\tilde t_0$ such that $\ddot \theta(\tilde t_0)=0$, then by Lemma \ref{lem:monotonity0}(c), (a) or (c) must happen. Otherwise, $\ddot \theta(\tilde t)>0$ for all critical point $\tilde t$. If there exists two moments $\tilde t_1<\tilde t_2\in (-t_0,0)$ such that $\ddot \theta(\tilde t_i)>0$, $i=1,2$, then by continuity, there must be a moment $\hat t\in (\tilde t_1,\tilde t_2)$ such that $\dot \theta(\hat t)=0$ and $\ddot \theta(\hat t)\leq0$, which contradicts the argument above. Therefore, the number of the critical points is at most one, which is a local minimum. This implies (b).
\end{proof}


Now, we assume $0\in A_0$ and $z(0)=0$, i.e. the three-body collision occurs at moment $0$. Due to the same reason, it is sufficient to consider $z(t)$ with $z(-t_0)\in \mathbb{C}^+_*$ in $t\in[-t_0,0)$. Then we prove the following monotonicity for the argument $\theta(t)$.

\begin{Lem}\label{lem:monotonity=0-}
Given a $t_0\in(0,T)$ and an action minimizer $z(t)=r(t)e^{\theta(t) i}$ of $\mathcal A_{-T,0}$ on $\Omega^{-T,0}_{A,A_0}$. Assume $z(t)$ is smooth on $(-t_0,0)$, $z(0)=0$ and $z(-t_0)\in \mathbb{C}^{+}_*$, then one of the following situations must happen.
\begin{enumerate}
\item[$(a)$] the argument $\theta(t) \equiv 0$ on $[-t_0,0)$.
\item[$(b)$] the argument $\theta(t)$ is strictly decreasing on $[-t_0,0)$ and the limit angle $\theta^-_*=\lim_{t\rightarrow 0^-}\theta(t)$ exists in $[0,\pi)$.
\item[$(c)$] the argument $\theta(t) \equiv \pi$ on $[-t_0,0)$.
\end{enumerate}
\end{Lem}
\begin{proof}
Let $I=\{t\in(-t_0,0):\theta(t)\in(-\pi,0)\}$. We define $\hat \theta_1(t)=0$ on $I$ and $\hat \theta_1(t)=\theta(t)$ on $[-T,0)\setminus I$ and let $\hat z_1(t)=r(t)e^{i\hat \theta_1(t)}$. By Lemma \ref{lem:anglepotential}, $U(z,t)>U( \hat z_1,t)$ and $\dot {\hat\theta}_1(t)=0$ on $I$. Similar to the proof of Lemma \ref{lem:monotonity0}(a), we obtain $\mathcal{A}_{-T,0}(z)>\mathcal{A}_{-T,0}(\hat z_1)$, which is a contradiction to minimizer $z$. Then $z(t)\in \mathbb{C}^+_*$ on $[-t_0,0)$.

Applying the approaches in Lemma \ref{lem:monotonity0} and Corollary \ref{cor:monotonity1}, we can similarly show that if both (a) and (c) do not happen on $(-t_0,0)$, then there exists $t_*\in[-t_0,0]$, such that $\theta(t)$ is strictly decreasing in $[-t_0,t_*]$ and strictly increasing in $[t_*,0]$. Define another path $\hat z_2(t)=r(t)e^{i\hat \theta_2(t)}$ with $\hat \theta_2(t)=\theta(t)$ on $[-T,t_*]$ and $\hat \theta_2(t)=\theta(t_*)$ on $[t_*,0]$. If $t_*<0$, by using Lemma~\ref{lem:anglepotential} again, $U(z,t)>U(\hat z_2,t)$ and $\dot{\hat\theta}_2=0$ on $(t_*,0)$. Then we have
\begin{align*}
\mathcal A_{-T,0}(z)-\mathcal A_{-T,0}(\hat z_2)
\geq  \int^{0}_{t_*}r^2\dot{\theta}^2(t) +U(z,t)-U(\hat z_2,t) dt >0.
\end{align*}
This leads to a contradiction to minimizer $z$ and implies (b), i.e. $t_*=0$. The proof is completed.
\end{proof}

Under the assumption $z(0)=0$, we can further prove the monotonicity of the argument for the velocity $\dot z(t)$. We first put $\dot{z}$ in the polar coordinates, $\dot{z}(t)=r_d(t)e^{\theta_d(t)i}$, with $r_d\in\mathbb R^+$ and $\theta_d\in [-\pi,\pi)$. Before we explore the monotonicity of $\theta_d$, we introduce the following lemma.

\begin{Lem}\label{lem:dot theta d equation}
The derivative of argument $\theta_d(t)$ is
\begin{equation}\label{eqn:dot theta_d}
\dot \theta_d(t)=\frac{1}{r_d(t)}\ddot z(t)\cdot e^{i(\theta_d(t)+\frac{\pi}{2})}.
\end{equation}
Moreover,
\begin{itemize}
\item[$(a)$] if $\theta_d(t)\in (\arg(\ddot z(t)),\arg(\ddot z(t))+\pi)$, then $\dot \theta_d(t)<0$.
\item[$(b)$] if $\theta_d(t)\in (\arg(\ddot z(t))-\pi,\arg(\ddot z(t)))$, then $\dot \theta_d(t)>0$.
\end{itemize}
\end{Lem}

\begin{proof}
Since $e^{i\theta_d(t)}=\frac{\dot z(t)}{r_d(t)}$, by differentiating $t$ on both sides, we have
$$\dot \theta_d(t)e^{i(\theta_d(t)+\frac{\pi}{2})}=\frac{\ddot z(t)}{r_d(t)}-\frac{\dot r_d(t)}{r_d(t)^2}\dot z(t).$$
Then taking the inner product with $e^{i(\theta_d(t)+\frac{\pi}{2})}$ on both sides, we obtain (\ref{eqn:dot theta_d}). Moreover, (a) and (b) follow from the fact that $\dot \theta_d(t)<0$ (resp. $>0$) if and only if
$$\theta_d(t)+\frac{\pi}{2}\in \left(\arg\ddot z(t)+\frac{\pi}{2}, \ \arg\ddot z(t)+\frac{3\pi}{2}\right)\ \left(\mathrm{resp.} \left(\arg\ddot z(t)-\frac{\pi}{2},\ \arg\ddot z(t)+\frac{\pi}{2}\right) \right).$$
\end{proof}

Now we are ready to prove the following lemma.
\begin{Lem}\label{lem:dot theta d-}
Given a $t_0\in(0,T)$ and an action minimizer $z(t)=r(t)e^{\theta(t) i}$ of $\mathcal A_{-T,0}$ on $\Omega^{-T,0}_{A,A_0}$. Assume $z(t)$ is smooth on $(-t_0,0)$, $z(0)=0$, and $z(-t_0)\in \mathbb{C}^+_*$. Then the argument $\theta_d(t)$ satisfies the following properties:
\begin{enumerate}
\item[$(a)$] if $\theta(-t_0)=0$, then there exists an $\epsilon\in(0,t_0)$ such that the argument $\theta_d(t) \equiv \pi$ for all $t \in (-\epsilon,0)$.
\item[$(b)$] if $\theta(-t_0)\in(0,\pi)$, there exists an $\epsilon\in (0,t_0)$ such that the argument $\theta_d(t)$ is strictly decreasing on $(-\epsilon,0)$.
\item[$(c)$] if $\theta(-t_0)=\pi$, then the argument $\theta_d(t)$ either satisfies $\theta_d(t)\in\{0,\pi\}$ on $\{t\in(-t_0,0): \dot z(t)\neq 0\}$ or there exists an $\epsilon\in (0,t_0)$ such that $\theta_d(t)$ is strictly decreasing on $(-\epsilon,0)$.
\end{enumerate}
Moreover, if $\theta(t)\not\equiv\pi$ on $(-t_0,0)$, then the limit $\lim_{t \rightarrow 0^-} \lvert \theta(t)-\theta_d(t) \rvert =\pi$ for any $\theta(-t_0)\in[0,\pi]$.
\end{Lem}
\begin{proof}
Firstly, (a) follows from Lemma \ref{lem:monotonity=0-} and the assumption $z(0)=0$ directly.

To prove (b). By Lemma \ref{lem:dot theta d equation} (a),
it is sufficient to show that there exists an $\epsilon>0$ sufficiently small such that $\theta_d(t)\in (\arg(\ddot z(t)),\arg(\ddot z(t))+\pi)$ for all $t \in (-\epsilon,0)$.
From Lemma~\ref{lem:monotonity=0-}$(b)$, $\theta(t)$ is strictly decreasing to $\theta^-_*\geq0$ on $(-t_0,0)$, which implies that
\begin{equation}\label{eqn:range of theta d1}
\theta(t)\in(0,\pi)\quad\mathrm{and}\quad \theta_d(t)\in (\theta(t)-\pi,\theta(t)),\ \forall t\in(-t_0,0).
\end{equation}
We see that $y(t)>0$ ($z(t)=x(t)+iy(t)$) for any $t\in(-t_0,0)$. Since the force is always pointing downward, then we have
\begin{equation}\label{eqn:range of force}
\arg\ddot z(t)\in(-\pi,\theta(t)-\pi),\ \forall t\in(-t_0,0).
\end{equation}
Moreover, there exists an $\epsilon\in(0,t_0)$ such that $y(t)$ is strictly decrease to $0$ during $t\in(-\epsilon,0)$, i.e.
\begin{equation}\label{eqn:range of theta d2}
\theta_d(t)\in (-\pi,0),\ \forall t\in(-\epsilon,0),
\end{equation}
Combining (\ref{eqn:range of theta d1}), (\ref{eqn:range of force}) and (\ref{eqn:range of theta d2}), we have
$$\theta_d(t)\in (\theta(t)-\pi,0)\subset (\arg\ddot z(t),\arg\ddot z(t)+\pi),\ \forall t\in[-\epsilon,0),$$
and prove that $\dot \theta_d(t)<0$ on $(-\epsilon,0)$.

To prove (c). By Lemma~\ref{lem:monotonity=0-}, $\theta(t)$ is either identically $\pi$ or strictly decreasing to $\theta_*^-\geq0$. The former case implies that $\theta(t)\in\{0,\pi\}$ on $\{t\in(-t_0,0):\dot z\neq 0\}$. For the latter case, we have $\theta(t)\in(0,\pi)$ for any $t\in (-t_0,0)$. Then the proof is similar to the case (b).

To see $\lim_{t\rightarrow 0^-}|\theta(t)-\theta_d(t)|=\pi$. According to the assumption $z(0)=0$, for any sequence of moments $t_k\rightarrow 0^-$, there exists a sequence $t_k'\in(t_k,0)$ such that
$$\dot z(t_k')= \lvert t_k \rvert ^{-1}(z(0)-z(t_k))=- \lvert t_k \rvert ^{-1}z(t_k)=\frac{r(t_k)}{ \lvert t_k \rvert }e^{i(\theta(t_k)+\pi)}.$$
By Lemma \ref{lem:monotonity=0-}, the argument $\theta(t)$ is always non-increasing and converging to $\theta_*^-\geq 0$ as $t\rightarrow 0^-$, then we have $\lim_{k\rightarrow +\infty}\theta_d(t_k')=\lim_{k\rightarrow +\infty}\theta(t_k)+\pi=\theta_*^-+\pi$. Finally, by the assumption $\theta(t)\not\equiv\pi$ and $(a)-(c)$, there exists an $\epsilon\in(0,t_0)$ such that the argument $\theta_d(t)$ is monotonic in $(-\epsilon,0)$, then we have
$\lim_{t\rightarrow 0^-}\theta_d(t)=\lim_{k\rightarrow +\infty}\theta_d(t_k')=\theta_*^-+\pi$. The proof is now complete.
\end{proof}

From the previous lemmas, we obtain some monotonicities for the minimizer of $\mathcal A_{-T,0}$ on $\Omega^{-T,0}_{A,A_0}$. By the reversibility of \eqref{eqn:1+1+1-body}, we have the following analog results for the minimizer of $\mathcal A_{0,T}$ on $\Omega^{0,T}_{A_0,A}$.

\begin{Lem}\label{lem:monotonity11}
Given a $t_0\in(0,T)$ and an action minimizer $z(t)=r(t)e^{\theta(t) i}$ of $\mathcal A_{0,T}$ on $\Omega^{0,T}_{A_0,A}$. Assume $z(t)\in\mathbb{C}^+_*$ is smooth on $(0,t_0)$, then one of the following situations must happen.
\begin{enumerate}
\item[$(a)$] the argument $\theta(t) \equiv 0$ on $[0,t_0]$.
\item[$(b)$] there is a unique $t_* \in [0,t_0]$ such that $\theta(t_*)\geq 0$, $\theta$ is strictly decreasing on $[0,t_0]$ and strictly increasing on $[t_*,t_0]$.
\item[$(c)$] the argument $\theta(t) \equiv \pi$ on $[0,t_0]$.
\end{enumerate}
\end{Lem}

\begin{Lem}\label{lem:monotonity=0+}
Given a $t_0\in(0,T)$ and an action minimizer $z(t)=r(t)e^{\theta(t) i}$ of $\mathcal A_{0,T}$ on $\Omega^{0,T}_{A_0,A}$. Assume $z(t)$ is smooth on $(0,t_0)$, $z(0)=0$ and $z(t_0)\in \mathbb{C}^+_*$, then one of the following situations must happen.
\begin{enumerate}
\item[$(a)$] the argument $\theta(t) \equiv 0$ on $(0,t_0]$.
\item[$(b)$] the argument $\theta(t)$ is strictly increasing on $(0,t_0]$ and the limit angle $\theta^+_*=\lim_{t\rightarrow 0^+}\theta(t)$ exists and non-negative.
\item[$(c)$] the argument $\theta(t) \equiv \pi$ on $(0,t_0]$.
\end{enumerate}
\end{Lem}

\begin{Lem}\label{lem:dot theta d+}
Given a $t_0\in(0,T)$ and an action minimizer $z(t)=r(t)e^{\theta(t) i}$ of $\mathcal A_{0,T}$ on $\Omega^{0,T}_{A_0,A}$. Assume $z(t)$ is smooth on $(0,t_0)$, $z(0)=0$, and $z(t_0)\in \mathbb{C}^+_*$. Then the argument $\theta_d(t)$ satisfies the following properties:
\begin{enumerate}
\item[$(a)$] if $\theta(t_0)=0$, then the argument $\theta_d(t) \equiv \pi$ for all $t \in (0,t_0)$.
\item[$(b)$] if $\theta(t_0)\in(0,\pi)$, there exists an $\epsilon\in (0,t_0)$ such that the argument $\theta_d(t)$ is strictly increasing on $(0,\epsilon)$.
\item[$(c)$] if $\theta(t_0)=\pi$, then the argument $\theta_d(t)$ is either $\theta_d(t)\equiv 0$ on $(0,t_0)$ or there exists an $\epsilon\in (0,t_0)$ such that $\theta_d(t)$ is strictly increasing on $(0,\epsilon)$.
\end{enumerate}
Moreover, the limit $\lim_{t \rightarrow 0^+} \lvert \theta(t)-\theta_d(t) \rvert =\pi$ for any $\theta(t_0)\in[0,\pi]$.
\end{Lem}

\section{Asymptotic properties of minimizer near the three-body collision}\label{sec3}
Given $T>0$ and a collision Kepler system $(q,c)$ which satisfies \eqref{eqn:newton} and $(Q1)-(Q3)$. Since the restricted one-center-two-body problem \eqref{eqn:1+1+1-body} is reversible, it is sufficient to consider our problem on $[-T,0]$. Let $z(t)=r(t)e^{i\theta(t)}$ be the action minimizer of $\mathcal{A}_{-T,0}$ on $\Omega^{-T,0}_{A,A_0}$ with $0\in A_0$. In this section, we aim to prove the Theorem \ref{main1}, i.e. the Sundman-Sperling estimates of the minimizer $z$ near the three-body collision. The proof includes the following three steps:
\begin{itemize}
\item By using the technique of critical and infliction points, we first show that the ratio $a(t)=r(t)/|q(t)|$ admits both positive upper and lower bound, see Section
    \ref{subsec: 3.1}.

\item By using the properties of $a(t)$, we prove that the asymptotic limit $\theta_*^-$ of $\theta(t)$ can only be $0$ or $\pi$ near the three-body collision, see Section \ref{subsec: 3.2}.

\item Based on the result $\theta_*^-\in\{0,\pi\}$ above, we can further improve the estimates of $a(t)$ by using the technique of critical and infliction points again. This enhanced estimates are enough for us to complete the proof of Theorem \ref{main1}, see Section \ref{subsec:sundman}.
\end{itemize}

Throughout this section, we assume $z(0)=0$. Without loss of generality, we further assume $z(t)$ is smooth on $[-t_0,0)$ and $z(-t_0)\in \mathbb{C}^+_*$ for some $t_0\in(0,T]$, since $\mathcal{A}_{-T,0}$ is invariant under the complex conjugation.

\subsection{Asymptotic behavior of $|z|$ near the three-body collision}\label{subsec: 3.1}
In this section, we aim to prove Theorem \ref{thm:main 3.1}, i.e. the ratio $a(t)=r(t)/|q(t)|$ is both bounded from above and below by two positive numbers near $t=0$.


\begin{Thm}\label{thm:main 3.1}
Let $z(t)=r(t)e^{i\theta(t)}$ be an action minimizer of $\mathcal{A}_{-T,0}$ on $\Omega^{-T,0}_{A,A_0}$. If $z(0)=0$, then
there exist an $\epsilon\in(0,T)$ and $0<c_a<C_a<+\infty$ such that $c_a<a(t)<C_a$ for all $t\in(-\epsilon,0)$.
\end{Thm}


The following lemma is immediately obtained from Proposition \ref{pro:sundmanq-}, and provides more convenience to the later discussions.

\begin{Lem}\label{lem:sundman}
There exist a $\delta>0$ small and $0<c_q<C_q<+\infty$ such that for any $t \in (-\delta,0)$, we have
\begin{align}\label{eqn:sundman}
 \lvert q(t) \rvert \in (c_q \lvert t \rvert ^{\frac{2}{3}},C_q \lvert t \rvert ^{\frac{2}{3}}),  \quad
 \lvert \dot q(t) \rvert \in (c_q \lvert t \rvert ^{-\frac{1}{3}},C_q \lvert t \rvert ^{-\frac{1}{3}}) \quad \text{and} \quad
 \lvert \ddot q(t) \rvert \in (c_q \lvert t \rvert ^{-\frac{4}{3}},C_q \lvert t \rvert ^{-\frac{4}{3}}).
\end{align}
\end{Lem}

We first show the positive upper bound for $a(t)$ near $t=0$. \begin{Lem}\label{lem: upper bound of r}
There exist $\delta>0$ and $\hat C>0$ such that $r(t)\leq \hat C |t|^{\frac{2}{3}}$ for all $t \in (-\delta,0)$. Moreover, $a(t)<\hat C/c_q$ for all $t\in(-\delta,0)$.
\end{Lem}
\begin{proof}
We prove this lemma in two situations: $\theta(t)\equiv\pi$ and $\theta(t)\not\equiv\pi$ for any $t\in[-t_0,0)$.

To the former case, if $r(-t_0)\in(0,|q(-t_0)|)$, then by Lemma \ref{lem:sundman}, we have $r(t)\leq |q(t)|<C_q|t|^{\frac{2}{3}}$ for all $t\in[-t_0,0)$, then this lemma follows.

If $r(-t_0)\in(|q(-t_0)|,+\infty)$, then we define a $(\hat z, \hat q)$-system, where $\hat q(t):=q(t)$ on $(-t_0,0)$ with mass $m+\mu$, and $\hat z(t)$ satisfies $\dot {\hat z}(-t_0)=\dot z(-t_0)$ and $\hat z(0)=0$. Since $(\hat z,\hat q)$ also forms a Kepler system, relatively, according to Proposition \ref{pro:sundmanq-}, we have $$|\hat z(t)|=|\hat z(t)-q(t)|+|q(t)|<(\lambda_{m+\mu}+\lambda_{\mu})|t|^{2/3}+o(|t|^{2/3}),$$ where $\lambda_{m+\mu}=(9(m+\mu)/2)^{1/3}$. We claim that $|\hat z(t)|>|z(t)|$ for any $t\in[-t_0,0)$. Otherwise, assume there exists a moment $\hat t\in[-t_0,0)$ such that $\hat z(\hat t)=z(\hat t)$ and $\dot{\hat z}(\hat t)\geq \dot z(\hat t)$. Since the total force of $\hat z$ in $(\hat z,\hat q)$-system is always greater than the total force of $z$ in $(q,c)$-system, whenever they are at the position, then $\hat z(t)$ reaches $0$ earlier than $z(t)$. This contradicts to the assumption that $\hat z(0)=0$.
Therefore, this claim holds. Finally, combing the arguments above, we conclude that
$|z(t)|<|\hat z(t)|<(\lambda_{m+\mu}+\lambda_{\mu})|t|^{2/3}+o(|t|^{2/3})$ for any $t\in (-t_0,0)$.

To prove the latter case. Fix any $t_1\in(0,t_0)$. Since $\theta(t)\not \equiv \pi$, then by Lemma \ref{lem:monotonity=0-} (b), $\theta(-t_1)\in(\theta_*^-,\pi)$ and $\theta(t)\leq \theta(-t_1)$ for any $t\in(-t_1,0)$. Since $z\cdot \dot z=r\dot r$, then $r\ddot r-z\cdot \ddot z= \lvert \dot z \rvert ^2-\dot r^2\geq0$, which implies that $r\ddot r\geq z\cdot \ddot z\geq -r \lvert \ddot z \rvert $.
By (\ref{eqn:1+1+1-body}), we have
\begin{equation}\label{eqn:-ddot r leq}
-\ddot r(t) \leq  \lvert \ddot z(t) \rvert \leq \frac{\mu}{ \lvert z(t) \rvert ^2}+\frac{m}{ \lvert z(t)-q(t) \rvert ^2}\leq 2\max\left\{\frac{\mu}{ \lvert z(t) \rvert ^2},\frac{m}{ \lvert z(t)-q(t) \rvert ^2}\right\}.
\end{equation}

If $r(t)\leq \lvert q(t) \rvert $, then we obtain the desired result immediately by Lemma \ref{lem:sundman}.

If $r(t)> \lvert q(t) \rvert $, then due to the facts that $a^2(t)+1+2a(t)\cos\theta(t) \geq 1 $ for any $\theta(t)\in [0,\pi/2]$ and
$a^2(t)+1+2a(t)\cos\theta(t) \geq \sin^2\theta(-t_1)>0$ for any $\theta(t) \in (\pi/2,\theta(-t_1))$, providing $\theta(-t_1)>\pi/2$. Then we have
\begin{equation}\label{eqn:lower bound of (a^2+1+2acos)}
\frac{m}{ \lvert z(t)-q(t) \rvert ^2}
=\frac{m}{(a(t)^2+1+2a(t)\cos\theta(t)) \lvert q(t) \rvert ^2}
\leq \frac{m}{\sin^2\theta(-t_1) \lvert q(t) \rvert ^2}.
\end{equation}
According to Lemma~\ref{lem:sundman}, (\ref{eqn:-ddot r leq}) and (\ref{eqn:lower bound of (a^2+1+2acos)}), there exists a $\delta\in(0,t_0)$ such that
\begin{equation}\label{eqn:-ddot r leq 2}
-\ddot r(t)\leq C|t|^{-\frac{4}{3}},\quad \forall t\in(-\delta,0),
\end{equation}
where $C=2\max\{\mu/c_q^2,\ m/(c_q\sin\theta(-t_1))^2, \ C_q\}$.

Denote $\Delta:=\{t\in(-\delta,0),\ r(t)> \lvert q(t) \rvert \}$.
Since $\Delta\subset (-\delta,0)$ is open, then we can write $\Delta=\cup_{k\geq 1}(a_k,b_k)$, where $r(a_k)= \lvert q(a_k) \rvert $ and $r(b_k)= \lvert q(b_k) \rvert $. Since $r(a_k)= \lvert q(a_k) \rvert $ and $r(t)> \lvert q(t) \rvert$ for any $t\in(a_k,b_k)$ and $k$, by $(\ref{eqn:sundman})$, we have
\begin{align}\label{eqn:-dot r a_k}
0<-\dot r(a_k)\leq -\frac{d}{dt} \lvert q(a_k) \rvert \leq C_q \lvert a_k \rvert ^{-\frac{1}{3}}.
\end{align}
By integration, for any $t\in(a_k,b_k)$, from $(\ref{eqn:-ddot r leq 2})$ and $(\ref{eqn:-dot r a_k})$, we compute that
\begin{eqnarray*}
r(t)&=&r(b_k)+\int^{b_k}_t \left( -\dot r(a_k)-\int^s_{a_k}\ddot r(\tau)d\tau \right)ds\\
&\leq& C_q \lvert b_k \rvert ^{\frac{2}{3}}+\int^{b_k}_t\left(C_q \lvert a_k \rvert ^{-\frac{1}{3}}+\int^s_{a_k}C \lvert \tau \rvert ^{-\frac{4}{3}}d\tau\right)ds\\
&=& C_q \lvert b_k \rvert ^{\frac{2}{3}}+(b_k-t)C_q \lvert a_k \rvert ^{-\frac{1}{3}}+\int^{b_k}_t3C\left( \lvert s \rvert ^{-\frac{1}{3}}- \lvert a_k \rvert ^{-\frac{1}{3}}\right)ds\\
&=& C_q \lvert b_k \rvert ^{\frac{2}{3}}+(b_k-t)C_q \lvert a_k \rvert ^{-\frac{1}{3}}+3C\left(\frac{3}{2}\left( \lvert t \rvert ^{\frac{2}{3}}- \lvert b_k \rvert ^{\frac{2}{3}}\right)-(b_k-t) \lvert a_k \rvert ^{-\frac{1}{3}}\right)\\
&=& \frac{9}{2}C \lvert t \rvert ^{\frac{2}{3}}
-\left(\frac{9}{2}C-C_q\right) \lvert b_k \rvert ^{\frac{2}{3}}
-\left(3C-C_q\right)(b_k-t) \lvert a_k \rvert ^{-\frac{1}{3}}\\
&<& \frac{9}{2}C \lvert t \rvert ^{\frac{2}{3}},
\end{eqnarray*}
the last inequality is obtained by $C\geq 2C_q$.
Hence, by taking $\hat C=9C/2$, the lemma follows.
\end{proof}


Next, we show that $a(t)$ also possesses a positive lower bound. We first introduce the following function
\begin{equation}\label{eqn:def of h}
h_{m/\mu}(a,\theta):=a^3-1-\frac{m}{\mu}\frac{a^2(a+\cos\theta)}{(a^2+1+2a\cos\theta)^{\frac{3}{2}}},
\quad \forall\ (a,\theta)\in \mathbb R^+\times\mathbb [0,\pi],\end{equation}
and study the relation between $h_{m/\mu}(a(t),\theta(t))$ and $\ddot{a}(t)$.

\begin{Lem}\label{lem: ddot a(t) and h(a,theta)}
Assume $\hat t\in(-t_0,0)$ is a critical point of $a(t)$, then $\ddot a(\hat t)>0$ (resp. $<0$) if and only if $$h_{m/\mu}(a(\hat t),\theta(\hat t))+\frac{1}{\mu}\tan^2(\eta(\hat t))\dot r^2(\hat t)r(\hat t)>0\ (\mathrm{resp. <0}),$$
where $\eta(t)$ denotes the angle from $z(t)$ to $\dot z(t)$. In particular, $\ddot a(\hat t)=0$ is equivalent to the case with equality.
\end{Lem}

\begin{proof}
By definition of $a(\cdot)$, we have
\begin{equation}\label{eqn:dot a and ddot a}
\dot a(\hat t)=\left(\dot r(\hat t)-a(\hat t)\frac{d}{dt} \lvert q(\hat t) \rvert \right)\frac{1}{ \lvert q(\hat t) \rvert }=0, \ \ \ddot a(\hat t)=\left(\ddot r(\hat t)-a(\hat t)\frac{d^2}{dt^2} \lvert q(\hat t) \rvert \right)\frac{1}{ \lvert q(\hat t) \rvert }.
\end{equation}
According to identities $ \lvert \dot z \rvert ^2=\dot r^2+r^2\dot\theta^2$ and $ \lvert \dot z \rvert ^2+z\cdot \ddot z={\dot r}^2+r\ddot r$, one can show that
\begin{align}\label{eqn:ddotrequation}
\ddot r=-\frac{\mu}{r^2}-\frac{mz\cdot(z-q)}{r \lvert z-q \rvert ^3}+{\dot\theta}^2r=-\frac{\mu}{r^2}-\frac{m(a+\cos\theta)}{(a^2+1+2a\cos\theta)^{\frac{3}{2}} \lvert q \rvert ^2}+{\dot\theta}^2r.
\end{align}
Then we have
\begin{equation}\label{eqn:ddotrh}
\begin{aligned}
\ddot r(t)-a(t)\frac{d^2}{dt^2} \lvert q(t) \rvert &
=-\frac{\mu}{r^2(t)}-\frac{m(a(t)+\cos\theta(t))}{(a^2(t)+1+2a(t)\cos\theta(t))^{\frac{3}{2}}\lvert q(t) \rvert ^2 }+\dot\theta^2(t)r(t)+a(t)\frac{\mu}{ \lvert q(t) \rvert ^2}\\
&=\left(a(t)-\frac{1}{a^2(t)}-\frac{m}{\mu}\frac{(a(t)+\cos\theta(t))}{(a^2(t)+1+2a(t)\cos\theta(t))^{\frac{3}{2}}}\right)\frac{\mu}{ \lvert q(t) \rvert ^2}+\dot\theta^2(t)r(t)\\
&=\left(h_{m/\mu}(a(t),\theta(t))+\frac{1}{\mu}\dot\theta^2(t)r^3(t)\right)\frac{\mu}{r^2(t)}\\
&=\left(h_{m/\mu}(a(t),\theta(t))+\frac{1}{\mu}\tan^2(\eta(t))\dot r^2(t)r(t)\right)\frac{\mu}{r^2(t)}.
\end{aligned}
\end{equation}
The last equality follows from $\tan\eta(t)=\frac{r(t)\dot\theta(t)}{\dot r(t)}$, since
$\cos(\eta(t)) \lvert z(t) \rvert  \lvert \dot{z}(t) \rvert =z(t)\cdot \dot{z}(t)=r(t)\dot{r}(t)$ and
$\sin(\eta(t)) \lvert z(t) \rvert  \lvert \dot{z}(t) \rvert =z(t)\times\dot{z}(t)=r^2(t)\dot{\theta}(t)$.
The lemma holds.
\end{proof}

The following lemma introduces the zero set $\mathcal{S}$ of $h_{m/\mu}(a,\theta)$. See Figure~\ref{fig:h} for the simulation of $\mathcal{S}$ with $m/\mu=1$.

\begin{Lem}\label{lem:zero set of h}
Given $h_{m/\mu}(a,\theta)$ as in (\ref{eqn:def of h}).  Let $\mathcal{S}:= \{(a,\theta): h_{m/\mu}(a,\theta)=0\}$,  we have
\begin{enumerate}
\item[$(a)$] if $\theta=\pi$, there exist unique two  $\alpha_1,\alpha_3$ with $\alpha_1<1<\alpha_3$ such that $(\alpha_1,\pi),(\alpha_3,\pi) \in \mathcal{S}$.
\item[$(b)$] if $\theta=0$,  there exists a unique $\alpha_2$ with $1<\alpha_2<\alpha_3$ such that $(\alpha_2,0)\in \mathcal{S}$.
\item[$(c)$] for each $a\in[\alpha_1,1)\cup[\alpha_2,\alpha_3]$, there are exactly one zero $\theta_{a}$ such that $(a,\theta_{a}) \in \mathcal{S}$.
\item[$(d)$] for any $a\in[0,\alpha_1)\cup[1,\alpha_2)\cup(\alpha_3,+\infty)$,  $(a,\theta) \notin \mathcal{S}$ for all $\theta \in [0,\pi]$.
\end{enumerate}
\end{Lem}

\begin{figure}\bigskip
\begin{center}

\includegraphics[width=0.7\textwidth]{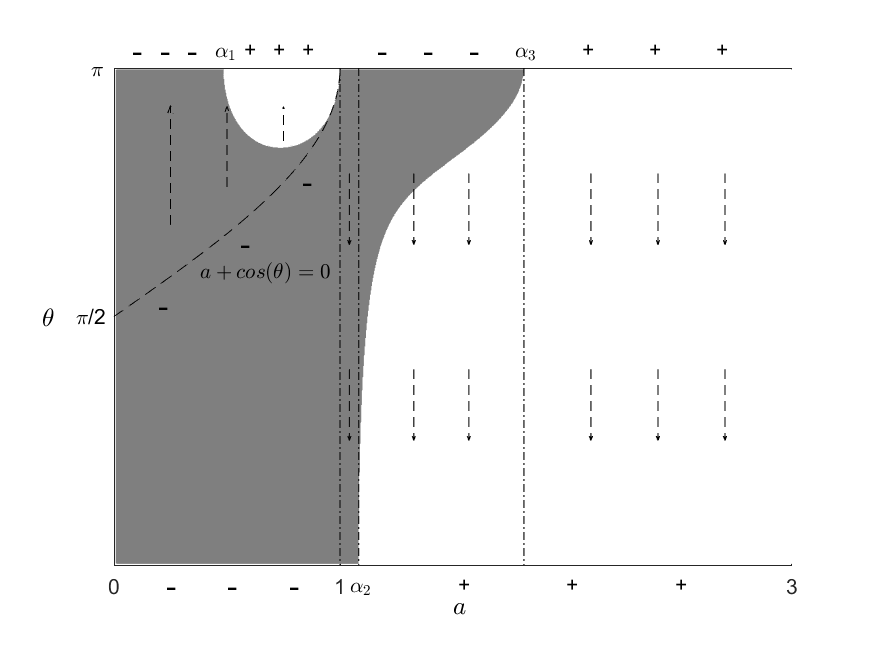}
\caption{
The white area and the gray area  indicate the collection of $(a,\theta)$ satisfying $h_{m/\mu}(a,\theta)>0$ and  $h_{m/\mu}(a,\theta)<0$ respectively,  and
the intersection curve between them is the zero set $\mathcal S$ of $h(a,\cdot)$.
The arrows are pointing to the direction in which $h(a,\cdot)$ increases.   Finally,  $h$ has a singular point $(1,\pi)$.
}

\label{fig:h}
\end{center}
\end{figure}

\begin{proof}
First of all,  we compute
the partial differential with respect to $\theta$ of function $h_{m/\mu}(a,\theta)$,
\begin{align}\label{eqn:partialangle}
\frac{\partial }{\partial \theta}h_{m/\mu}(a,\theta):=\frac{\frac{m}{\mu}a^2\sin\theta(1-2a^2-a\cos\theta)}{(a^2+1+2a\cos\theta)^{\frac{5}{2}}}.
\end{align}

Next,  according to the range of $a$,  we divide the proof into 3 cases: \\

\noindent  Case 1: $a=1$.

Since
$$h_{m/\mu}(1,\theta)=-\frac{m}{\mu}\frac{1+\cos\theta}{(2+2\cos\theta)^{\frac{3}{2}}}
=-\frac{m}{\mu}\frac{1}{2(2+2\cos\theta)^{\frac{1}{2}}}$$
is decreasing on $[0,\pi]$ from $-m/(4\mu)$ to $-\infty$ with respect to $\theta$,  we know that  $\{(a,\theta) \in \mathcal{S}:a=1\}=\emptyset$.

\noindent  Case 2:  $0 \leq a <1$.

We first assume  $(a_*,\theta_*) \in \mathcal{S}$ and $0\leq a_*<1$,
then $\theta_* \in (\pi/2,\pi]$ and $a_* \in (0,  -\cos \theta)$.
This implies the set $\{(a,\theta) \in \mathcal{S}: 0\leq a<1\} \subset D_1$ where $D_1$ denotes the area surrounded by the lines $a=0$,  $\theta=\pi$ and  curve $\{(a,\theta)\in[0,1]\times[0,\pi]:a=-\cos \theta\}$
(see the upper left part of Figure~\ref{fig:h}).

In the area $D_1$,  we have $1-2a^2-a\cos\theta>0$ and then,  by (\ref{eqn:partialangle}),  $\frac{\partial }{\partial \theta}h_{m/\mu}(a,\theta)>0$.
This means that,   when $a$ is fixed, the function $h_{m/\mu}(a,\cdot)$ is strictly increasing.

On the line $\theta=\pi$.  The equations
\begin{align*}
h_{m/\mu}(a,\pi)=a^3-1-\frac{m}{\mu}\frac{a^2}{(1-a)^2}\quad \ \text{ and } \ \quad
\frac{d}{da}h_{m/\mu}(a,\pi)=3a^2+\frac{m}{\mu}\frac{2a}{(1-a)^3}>0
\end{align*}
 imply that $h_{m/\mu}(\cdot,\pi)$ is strictly increasing on $[0,1]$ from $-1$ to $+\infty$. Then there is a unique $\alpha_1$ satisfying $h_{m/\mu}(\alpha_1,\pi)=0$,  $h_{m/\mu}(a,\pi)<0$ if $ a\in [0,\alpha_1)$ and $h_{m/\mu}(a,\pi)>0$ if $  a \in (\alpha_1,1)$.

On the curve $\{(a,\theta)\in[0,1]\times[0,\pi]:a=-\cos \theta\}$,  it is easy to check that $h_{m/\mu}(a,\theta)<0$.

According to the monotonicity of $h_{m/\mu}(a,\cdot)$ and the signs of $h_{m/\mu}(a,\theta)$ on line $\theta=\pi$ and on curve $\{(a,\theta)\in[0,1]\times[0,\pi]:a=-\cos \theta\}$,  we demonstrate that,  for each $a\in [\alpha_1,1)$,  there exists exactly one $\theta_a \in [\pi/2,\pi]$ satisfying $(a,\theta_a) \in S$ and,  for any $a \in [0,\alpha_1)$,   $(a,\theta_a) \notin S$ for all $\theta \in [0,\pi]$,  that is
$\{(a,\theta) \in \mathcal{S}:0\leq a <1\}=\{(a,\theta_a):\alpha_1\leq a <1\}$.

\noindent  Case 3:  $a >1$.

Let $D_2:=\{(a,\theta):1<a<+\infty,  0\leq \theta \leq \pi\}$,
similar to Case 2,  we discuss the monotonicity of $h_{m/\mu}(a,\cdot)$ in $D_2$ and the signs of $h_{m/\mu}(a,\theta)$ on lines $\theta=0$ and $\theta=\pi$.

In the area $D_2$,   we have $1-2a^2-a\cos\theta<0$ and then,  by (\ref{eqn:partialangle}),  $\frac{\partial }{\partial \theta}h_{m/\mu}(a,\theta)<0$.
This means that,  when $a$ is fixed,  the function $h_{m/\mu}(a,\cdot)$ is strictly decreasing.

On the line $\theta=0$.  The equations
\begin{align*}
h_{m/\mu}(a,0)=a^3-1-\frac{m}{\mu}\frac{a^2}{(a+1)^2}\quad \ \text{ and } \ \quad
\frac{d}{da}h_{m/\mu}(a,0)=\frac{2a}{(a+1)^3}\left(\frac{3}{2}a(a+1)^3-\frac{m}{\mu}\right)
\end{align*}
implies that $h_{m/\mu}(\cdot,0)$ is either strictly increasing from $-m/(4\mu)$ to $+\infty$ or strictly decreasing from $-m/(4\mu)$ to some negative value $-C<0$ and then  strictly increasing to $+\infty$. For both of these cases, there are only one $\alpha_2>1$ such that $h_{m/\mu}(\alpha_2,0)=0$, $h_{m/\mu}(a,0)<0$ if $ a\in (1,\alpha_2)$ and $h_{m/\mu}(a,0)>0$ if $  a \in (\alpha_2,+\infty)$.

On the line $\theta=\pi$.  The equations
\begin{align*}
h_{m/\mu}(a,\pi)=a^3-1-\frac{m}{\mu}\frac{a^2}{(a-1)^2}\quad \ \text{ and } \ \quad
\frac{d}{da}h_{m/\mu}(a,\pi)=3a^2+\frac{m}{\mu}\frac{2a}{(a-1)^3}>0
\end{align*}
 imply that $h_{m/\mu}(\cdot,\pi)$ is strictly increasing on $(1,+\infty)$ from $-\infty$ to $+\infty$.
Then there is a unique $\alpha_3$ satisfying $h_{m/\mu}(\alpha_3,\pi)=0$,  $h_{m/\mu}(a,\pi)<0$ if $ a\in (1,\alpha_3)$ and $h_{m/\mu}(a,\pi)>0$ if $  a \in (\alpha_3,+\infty)$.
Note that $\alpha_2<\alpha_3$ because $h_{m/\mu}(\alpha_2,\pi)<h_{m/\mu}(\alpha_2,0)=0$.

Summarize the above discussions,
we demonstrate that,  for each $a\in [\alpha_2,\alpha_3]$,  there exists exactly one $\theta_a \in [0,\pi]$ satisfying $(a,\theta_a) \in S$ and,  for any $a \in (1,\alpha_2)\cup(\alpha_3,+\infty)$,   $(a,\theta) \notin S$ for all $\theta \in [0,\pi]$
(see the right part of Figure~\ref{fig:h}),
that is
$\{(a,\theta) \in \mathcal{S}:a >1\}=\{(a,\theta_a):\alpha_2\leq a \leq \alpha_3\}$.
This completes the proof.
\end{proof}

By using Lemma \ref{lem:zero set of h}, we further obtain the following properties for $h_{m/\mu}(a,\theta)$.
\begin{Lem}\label{lem: zeros of h(a,theta)}
Given $\alpha_1,\alpha_2,  \alpha_3$ as in Lemma~\ref{lem:zero set of h}. Then we have
\begin{itemize}
\item[(a)] if $a>\alpha_3$,  $h_{m/\mu}(a,\theta)>0$ for any $\theta\in[0,\pi]$.
\item[(b)] if $a<\alpha_1$,  $h_{m/\mu}(a,\theta)<0$ for any $\theta\in[0,\pi]$.
\item[(c)] if $a<\alpha_2$ (resp.  $a>\alpha_2$),  $h_{m/\mu}(a,0)<0$ (resp.  $h_{m/\mu}(a,0)>0$).
\item[(d)] if $a\in[0,\alpha_1)\cup (1,\alpha_3)$ (resp.  $a\in(\alpha_1,1)\cup(\alpha_3,+\infty)$),  $h_{m/\mu}(a,\pi)<0$ (resp.  $h_{m/\mu}(a,\pi)>0$).
\end{itemize}
\end{Lem}

Based on the results above, we now delve into the study of asymptotic behavior of $a(t)$ as $t\rightarrow 0^-$.
Let $\underline{a}^-=\liminf_{t\rightarrow 0^-} a(t)$ and $\bar a^-=\limsup_{t\rightarrow 0^-}a(t)$.

\begin{Prop}\label{prop: convergency of a}
Given $\alpha_1$ as in Lemma~\ref{lem:zero set of h},  we have
\begin{itemize}
\item[(a)] if $\bar a^- \leq \alpha_1$, then $\bar a^-=\underline{a}^-=:a^*$, i.e. $a(t)$ is convergent as $t\rightarrow 0^-$.
\item[(b)] if $\bar a^->\alpha_1$, then $\underline{a}^-\geq \alpha_1$.
\end{itemize}
\end{Prop}

\begin{proof}
To show (a). Assume $\bar a^->\underline{a}^-$.
From $\underline{a}^-<\alpha_1$,   there exists a sequence $t_k\rightarrow 0^-$ such that
\begin{equation}\label{eqn:moment tk}
a(t_k)<\alpha_1,\ \ \dot a(t_k)=0 \ \text{ and }\  \ddot a(t_k)\geq0.
\end{equation}
By (\ref{eqn:dot a and ddot a}) and Lemma \ref{lem: ddot a(t) and h(a,theta)}, (\ref{eqn:moment tk}) implies that $\dot r(t_k)=a(t_k)\frac{d}{dt} \lvert q(t_k) \rvert $ and
\begin{equation}\label{eqn:moment tk 1}
h_{m/\mu}(a(t_k),\theta(t_k))+\frac{\tan^2\eta(t_k)}{\mu}a^3(t_k)\left(\frac{d}{dt} \lvert q(t_k) \rvert \right)^2 \lvert q(t_k) \rvert \geq0,
\end{equation}
where $\eta(t)$  denotes  the  angle  from  $z(t)$  to  $\dot{z}(t)$.
Up to a subsequence, we assume the sequence
$\{a(t_k)\}$ is convergent and $ \hat a:=\lim_{k\rightarrow +\infty}a(t_k) \leq\frac{\alpha_1+\underline a^-}{2}<\alpha_1$.
Then by Lemma \ref{lem: zeros of h(a,theta)}(b), we have
\begin{align*}
\lim_{k\rightarrow +\infty}h_{m/\mu}(a(t_k),\theta(t_k)) = h_{m/\mu}(\hat a,\theta_*^-)<-\epsilon,
\end{align*}
for some $\epsilon>0$ sufficiently small.
Since $a(t_k)$ and
$$\left(\frac{d}{dt} \lvert q(t_k) \rvert \right)^2 \lvert q(t_k) \rvert \leq C_q^2 \lvert t_k \rvert ^{-\frac{2}{3}}C_q \lvert t_k \rvert ^{\frac{2}{3}}=C_q^3$$
are bounded by $(\ref{eqn:sundman})$ and $(\ref{eqn:moment tk})$,
and $\tan^2\eta(t_k)$ is converging to $0$ by Lemma~\ref{lem:dot theta d-},  the left side of (\ref{eqn:moment tk 1}) is negative when $k$ is large enough, which is a contradiction. Hence, (a) follows.

To show (b).
If $\bar a^->\alpha_1$,  but $\underline a^-<\alpha_1$,  then there also exist a sequence $t_k\rightarrow 0^-$ with satisfying (\ref{eqn:moment tk}). By following the argument in (a),  we obtain a  similar contradiction. The proof is completed.
\end{proof}
Now we are ready to prove Theorem \ref{thm:main 3.1}.


\begin{proof}[Proof of Theorem \ref{thm:main 3.1}]
According to Lemma \ref{lem: upper bound of r} and Proposition \ref{prop: convergency of a}, it is sufficient to prove $a^*>0$, where $a^*=\lim_{t\rightarrow 0^-}a(t)$ providing $\bar a^-=\underline a^- \leq \alpha_1$. By contradiction, we assume $a^*=0$. The proof is separated in two cases: $\theta(t)\not\equiv \pi$ and $\theta(t)\equiv \pi$ on $[-t_0,0)$.

To the former case, we first introduce the following two claims.

\noindent \textbf{Claim 1}: There exists a sequence $\{t'_k\}$ with $t'_k \rightarrow 0^-$ as $k\rightarrow +\infty$ such that $0< \lvert \dot r(t'_k) \rvert <\hat C \lvert t'_k \rvert ^{-\frac{1}{3}}$, for all  $k$,  where $\hat C>0$ is given in Lemma~\ref{lem: upper bound of r}.

Indeed. Recall that $\eta(t)$ is the angle between $z(t)=r(t)e^{i\theta(t)}$ and $\dot z(t)=r_d(t)e^{i\theta_d(t)}$.
From Lemma~\ref{lem:monotonity=0-} and Lemma~\ref{lem:dot theta d-},
as $t \rightarrow 0^-$,  we have
\begin{equation}\label{eqn:dot r<0}
\begin{aligned}
&\eta(t)\rightarrow \pi,\quad \tan\eta(t)\rightarrow0\ \mathrm{and}\  \dot r(t)<0, \quad \text{ if } \theta(t)\not \equiv 0; \\
&\eta(t)\equiv \pi,\quad  \tan\eta(t)=0\ \mathrm{and}\  \dot r(t)<0, \quad \text{ if } \theta(t)\equiv0.
\end{aligned}
\end{equation}
Given $\delta,\hat C>0$ as in Lemma~\ref{lem: upper bound of r} and  a convergent sequence $\{t_k\}\subset(-\delta,0)$ with $t_k \rightarrow 0^-$ as $k \rightarrow +\infty$.
For each $k$,  by mean value theorem,  there exists a moment  $t'_k \in (t_k,0)$ such that
\begin{align}\label{eqn:dotr<estimate}
 \lvert \dot{r}(t'_k) \rvert  = \frac{ \lvert r(t_k)-r(0) \rvert }{ \lvert t_k-0 \rvert } = \frac{ \lvert r(t_k) \rvert }{ \lvert t_k \rvert }
<\hat{C} \lvert t_k \rvert ^{-\frac{1}{3}}<\hat{C}\lvert t'_k \rvert ^{-\frac{1}{3}}.
\end{align}
By (\ref{eqn:dot r<0}) and (\ref{eqn:dotr<estimate}),  $\{t'_k\}$ is the desired sequence in Claim 1.

\noindent \textbf{Claim 2}: There exists an $\epsilon>0$ small, such that $0< \lvert \dot r(t) \rvert <\hat C\lvert t \rvert ^{-\frac{1}{3}}$, for all  $t\in(-\epsilon,0)$.

Write $\beta(t)= \lvert \dot r(t) \rvert \cdot \lvert t \rvert ^{\frac{1}{3}}$. It is sufficient to prove $\beta(t)<\hat C$ for any $t\in (-\epsilon,0)$.
If not,  there is a sequence of moments $\bar t_k\rightarrow 0^-$
such that $\beta(\bar t_k)\geq \hat C$.
For each $t'_k$ as in Claim~1, there exists $j_k$ such that
$t_k'\in(\bar t_{j_k},\bar t_{j_k+1})$,
 then the inequalities $\beta(t_k')<\hat C$,  $\beta(\bar t_{j_k})\geq\hat C$ and $\beta(\bar t_{j_k+1})\geq\hat C$ imply that there is a moment $s_k\in(\bar t_{j_k},\bar t_{j_k+1})$ such that $\beta(s_k)<\hat C$ and $\dot \beta(s_k)=0$.
On the one hand, we can conclude that
\begin{equation}\label{eqn:estimates of ddot r}
\ddot r(s_k)=-\dot \beta(s_k) \lvert s_k \rvert ^{-\frac{1}{3}}-\frac{1}{3}\beta(s_k) \lvert s_k \rvert ^{-\frac{4}{3}}=-\frac{1}{3}\beta(s_k) \lvert s_k \rvert ^{-\frac{4}{3}}
>-\frac{1}{3}\hat C \lvert s_k \rvert ^{-\frac{4}{3}}.
\end{equation}
On the other hand, combining (\ref{eqn:1+1+1-body}),
$ \lvert \dot z \rvert ^2=\dot r^2+\dot\theta^2r^2$,  $ \lvert \dot z \rvert ^2+z\cdot \ddot z={\dot r}^2+r\ddot r$ and $\tan\eta(t)=\frac{r(t)\dot\theta(t)}{\dot r(t)}$, one can show that
\begin{equation}\label{eqn:equation of ddot r}
\ddot r=-\frac{\mu}{r^2}-\frac{mz\cdot(z-q)}{r \lvert z-q \rvert ^3}+{\dot\theta}^2r
=-\frac{\mu}{r^2}-\frac{m(a+\cos\theta)}{(a^2+1+2a\cos\theta)^{\frac{3}{2}} \lvert q \rvert ^2}+\tan^2\eta\frac{\dot r^2}{r}.
\end{equation}
By using (\ref{eqn:equation of ddot r}), we can further conclude that,  for $k$ sufficiently large,
\begin{equation}\label{eqn:estimates of ddot r 1}
\begin{aligned}
\ddot r(s_k)&=\left(-\frac{\mu}{a^2(s_k)}-\frac{m(a(s_k)+\cos\theta(s_k))}{\left(a^2(s_k)+1+2a(s_k)\cos\theta(s_k)\right)^{\frac{3}{2}}}
+\tan^2\eta(s_k)\frac{\beta^2(s_k) \lvert q(s_k) \rvert }{a(s_k) \lvert s_k \rvert ^{\frac{2}{3}}}\right)\frac{1}{ \lvert q(s_k) \rvert ^2}\\
&\leq \left(-\frac{\mu}{a^2(s_k)}+\frac{m}{ \lvert a^2(s_k)+1+2a(s_k)\cos\theta(s_k) \rvert }
+\tan^2\eta(s_k)\frac{\hat C^2C_q}{a(s_k)}\right)\frac{1}{ \lvert q(s_k) \rvert ^2}\\
&<-\frac{\mu}{2a^2(s_k)}\frac{1}{ \lvert q(s_k) \rvert ^2}<-\frac{\mu}{2a^2(s_k)C_q^2} \lvert s_k \rvert ^{-\frac{4}{3}}.
\end{aligned}
\end{equation}
The last two inequalities above follow by
 (\ref{eqn:sundman}),  $\beta(s_k)<\hat C$, $\lim_{k \rightarrow +\infty} a(s_k) = a^*=0$ and $\lim_{t\rightarrow0^-}\tan\eta(t)=0$, since the term $-\mu a^{-2}(s_k)$ is the dominate term in (\ref{eqn:estimates of ddot r 1}), for $k$ sufficiently large.
From (\ref{eqn:estimates of ddot r}) and (\ref{eqn:estimates of ddot r 1}),  we observe
$$-\frac{1}{3}\hat C
<\ddot r(s_k) \lvert s_k \rvert ^{\frac{4}{3}}
<-\frac{\mu}{2a^2(s_k)C_q^2},
$$
which gives a contradiction since the right hand side is smaller than the left hand side, as $k \rightarrow +\infty$. Hence, Claim 2 holds.

Now to the latter case, i.e. $\theta(t)\equiv \pi$ on $[-t_0,0)$, we prove a similar claim as before.

\noindent\textbf{Claim 3}: There exists an $\epsilon>0$ small, such that $|\dot r(t)|<\hat C|t|^{\frac{1}{3}}$ for any $t\in(-\epsilon,0)$.

By contradiction, we assume there exists a sequence $\bar t_k\rightarrow 0^-$ such that $\beta(\bar t_k)=|\dot r(\bar t_k)|\cdot |\bar t_k|^{\frac{1}{3}}\geq \hat C$. Similar to Claim 1, by taking $\delta>0$ as in Lemma \ref{lem: upper bound of r} and a sequence $t_k\rightarrow 0^-$, there exists another sequence $t_k'\rightarrow 0$ such that $|\beta(t_k')|<\hat C$ for any $k$. Similar to Claim 2, due to the properties of $\bar t_k$ and $t_k'$, there also exists a sequence $s_k\rightarrow 0^-$ such that $\beta(s_k)<\hat C$ and $\dot \beta(s_k)=0$. On the one hand, by \eqref{eqn:estimates of ddot r}, we have $\ddot r(s_k)>-\frac{1}{3}\hat C|s_k|^{-\frac{4}{3}}$. On the other hand, since $a^*=0$, for any $t\in(-\epsilon,0)$ sufficiently small, by  \eqref{eqn:sundman}, we have
\begin{equation}\label{eqn:estimates of ddot r 2}
\ddot r=-\frac{\mu}{r^2}-\frac{mz\cdot(z-q)}{r \lvert z-q \rvert ^3}
=-\frac{\mu}{r^2}+\frac{m}{(1-a)^2|q|^2}
=(-\frac{\mu}{a^2}+\frac{m}{(1-a)^2})\frac{1}{|q|^2}\leq -\frac{\mu}{2a^2C_q^2}|t|^{-\frac{4}{3}}.
\end{equation}
This contradicts to $\ddot r(s_k)>-\frac{1}{3}\hat C|s_k|^{-\frac{4}{3}}$ for $k$ sufficiently large. Therefore, Claim 3 holds.

Now we complete the proof for both two cases. Since $a^*=\lim_{t\rightarrow 0^-}a(t)=0$, then there is a sequence $\{t_k\}$ with $t_k \rightarrow 0^-$ as $k \rightarrow +\infty$ such that $\dot a(t_k)<0$ and $a(t_k)\rightarrow 0$ as $k \rightarrow +\infty$,   then
$$\dot r(t_k)=\dot a(t_k) \lvert q(t_k) \rvert +a(t_k)\frac{d}{dt} \lvert q(t_k) \rvert <a(t_k)\frac{d}{dt} \lvert q(t_k) \rvert .$$
Write $f_k(t)=r(t)-a(t_k) \lvert q(t) \rvert $, then $f_k(0)=f_k(t_k)=0$ and $\dot f_k(t_k)<0$. Similar to (\ref{eqn:estimates of ddot r 1}) and (\ref{eqn:estimates of ddot r 2}), we have $\ddot r(t)<-\frac{\mu}{2a^2(t)C_q^2} \lvert t \rvert ^{-\frac{4}{3}}$. Together with (\ref{eqn:sundman}),
we have
$$\ddot f_k(t)=\ddot r(t)-a(t_k)\frac{d^2}{dt^2} \lvert q(t) \rvert \leq \left(-\frac{\mu}{2a^2(t)C_q^2}+a(t_k)C_q\right) \lvert t \rvert ^{-\frac{4}{3}}<0.$$
Furthermore, if $k$ is large enough, we have $\dot f_k(t)=\dot f_k(t_k)+\int^t_{t_k}\ddot f_k(s)ds<0,\ \forall t\in (t_k,0)$. This means that $f_k$ is strictly decreasing on $(t_k,0)$. However, this is impossible since $f_k(t_k)=f_k(0)=0$, contradiction. This completes the proof.
\end{proof}

\subsection{Asymptotic behavior of $\mathrm{arg}z$ near the three-body collision}\label{subsec: 3.2}

In this subsection, we devote our attention to proving Theorem~\ref{thm:main 3.2}.

Based on Lemma~\ref{lem:monotonity=0-}, we know that when the minimizer $z=re^{i\theta}$ experiences the three-body collision,  the asymptotic angle
$\theta_*^{-}:=\lim_{t\rightarrow0^{-}}\theta(t)$ exists.
In the following, we further reveal that the asymptotic angle $\theta_*^{-}$ must be $0$ or $\pi$.

\begin{Thm}\label{thm:main 3.2}
Let $z(t)=r(t)e^{i\theta(t)}$ be the action minimizer of $\mathcal{A}_{-T,0}$ on $\Omega^{-T,0}_{A,A_0}$. If $z(0)=0$, then the limit angle $\theta_*^-\in\{0,\pi\}$. More precisely, if $\theta(t)\not\equiv \pi$ (resp. $\theta(t)\equiv \pi$), then we always have $\theta_*^-=0$ (resp. $\theta_*^-=\pi$).
\end{Thm}

Firstly, we recall that $z(0)=0$, $z(t)$ is smooth on $[-t_0,0)$ and $z(-t_0)\in \mathbb{C}^+_*$ for some $t_0\in(0,T]$. By Lemma \ref{lem:monotonity=0-}, we have $\theta(t)\in[0,\pi]$ on $[-t_0,0]$. To prove this theorem, we need to estimate the angular momentum $J(t):=z(t)\times \dot z(t)=r^2(t)\dot{\theta}(t)$, which is non-positive since $\dot \theta(t)\leq 0$ on $t\in[-t_0,0)$.

\begin{Lem}\label{lem: estimates of J}
Assume $\theta_*^-\in(0,\pi)$. Then there exists an $\epsilon>0$ small and a pair of constants $c_J, C_J>0$ such that  $c_J |t|^{\frac{1}{3}} <|J(t)|< C_J |t|^{\frac{1}{3}}$ for all $t\in (-\epsilon,0)$.
\end{Lem}

\begin{proof}
To show the upper bounded of $J(t)$. According to Lemma~\ref{lem:sundman} -~\ref{lem: upper bound of r},
we can choose a $\delta>0$ which makes (\ref{eqn:sundman}) and Lemma~\ref{lem: upper bound of r} hold simultaneously.
By Lemma \ref{lem: upper bound of r},  we have
\begin{equation}\label{eqn:estimate of J}
J(t)^2=r^4(t)\dot\theta^2(t)\leq r^2(t) \lvert \dot z(t) \rvert ^2\leq \hat C^2 \lvert t \rvert ^{\frac{4}{3}} \lvert \dot z(t) \rvert ^2, \quad \text{ for all }t\in(-\delta,0).
\end{equation}
To estimate $ \lvert \dot{z}(t) \rvert ^2$,   we consider the following.
By (\ref{eqn:1+1+1-body}), for  $t\in(-\delta,0)$, we have
\begin{equation}\label{eqn:half of dot z square}
\begin{aligned}
\frac{1}{2} \lvert \dot{z}(t) \rvert ^2
&=\frac{1}{2} \lvert \dot{z}(-\delta) \rvert ^2+\int^{t}_{-\delta}\dot z(s)\cdot\ddot z(s)ds \\
&=\frac{1}{2} \lvert \dot{z}(-\delta) \rvert ^2+\int^t_{-\delta}\dot z(s)\cdot \left(-\frac{\mu z(s)}{ \lvert z(s) \rvert ^3}-\frac{m (z(s)-q(s))}{ \lvert z(s)-q(s) \rvert ^3}\right)ds\\
&=\frac{1}{2} \lvert \dot{z}(-\delta) \rvert ^2+\int^t_{-\delta}-\left(\dot z(s)\cdot \frac{\mu z(s)}{ \lvert z(s) \rvert ^3}\right)-\left(\left(\dot z(s)-\dot q(s)+\dot q(s)\right) \cdot\frac{m(z(s)-q(s))}{ \lvert z(s)-q(s) \rvert ^3}\right)ds\\
&=C_{1,\delta}+\frac{\mu}{r(t)}+\frac{m}{ \lvert z(t)-q(t) \rvert }-\int^t_{-\delta}\frac{m\cos \psi(s)\lvert \dot q(s) \rvert }{ \lvert z(s)-q(s) \rvert ^2}ds,
\end{aligned}
\end{equation}
where $C_{1,\delta}=\frac{1}{2} \lvert \dot{z}(-\delta) \rvert ^2-(\frac{\mu}{r(-\delta)}+\frac{m}{ \lvert z(-\delta)-q(-\delta) \rvert })$ and $\psi(s)$ denotes the angle between $z(s)-q(s)$ and $\dot q(s)$.
By Theorem~\ref{thm:main 3.1},  we know that $a(t)>c_a$ for some constant $c_a>0$.
Then, applying (\ref{eqn:sundman}) and (\ref{eqn:lower bound of (a^2+1+2acos)}), we have
\begin{equation}\label{eqn:middle term}
\begin{aligned}
\frac{\mu}{r(t)}+\frac{m}{ \lvert z(t)-q(t) \rvert }
&=\left(\frac{\mu}{a(t)}+\frac{m}{(a^2(t)+1+2a(t)\cos\theta(t))^{\frac{1}{2}}}\right)\frac{1}{ \lvert q(t) \rvert }\\
&<\left(\frac{\mu}{ c_qc_a}+\frac{m}{ c_q\sin \theta(-\delta)}\right) \lvert t \rvert ^{-\frac{2}{3}}.
\end{aligned}
\end{equation}
For the last term of (\ref{eqn:half of dot z square}),
by (\ref{eqn:sundman}) and (\ref{eqn:lower bound of (a^2+1+2acos)}) again,
we can further obtain that
\begin{equation}\label{eqn:last term}
\begin{aligned}
\left \lvert \int^t_{-\delta}\frac{m\cos \psi(t) \lvert \dot q(t) \rvert }{ \lvert z(t)-q(t) \rvert ^2}dt\right \rvert
&\leq\int^t_{-\delta}\frac{m \lvert \cos \psi(t) \rvert \left \lvert \dot q(t)\right \rvert  }{(a^2(t)+1+2a(t)\cos\theta(t)) \lvert q(t) \rvert ^2}dt \\
&\leq\int^t_{-\delta}\frac{mC_q}{c_q^2\sin^2\theta(-\delta)} \lvert t \rvert ^{-\frac{5}{3}}dt\\
&=\frac{3mC_q}{2c_q^2\sin^2 \theta(-\delta)} \lvert t \rvert ^{-\frac{2}{3}}-C_{2,\delta},
\end{aligned}
\end{equation}
where $C_{2,\delta}=\frac{3mC_q}{2c_q^2\sin^2 \theta(-\delta)} \lvert \delta \rvert ^{-\frac{2}{3}}$.

Then, combining (\ref{eqn:half of dot z square}), (\ref{eqn:middle term}) and (\ref{eqn:last term}), we obtain that
\begin{equation}\label{eqn:half of dot z square 1}
 \lvert \dot{z}(t) \rvert ^2\leq C_{3,\delta} \lvert t \rvert ^{-\frac{2}{3}}+C_{4,\delta},
\end{equation}
where $$C_{3,\delta}=2\left(\frac{\mu}{c_ac_q}+\frac{m}{c_q \sin \theta(-\delta)}\right)
+\frac{3mC_q}{c_q^2 \sin^2 \theta(-\delta)}>0  \quad \text{and }\quad
  C_{4,\delta}=2C_{1,\delta}-2C_{2,\delta}.$$
Therefore, by (\ref{eqn:estimate of J}) and (\ref{eqn:half of dot z square 1}), we have
\begin{equation}\label{eqn:estimate of J 1}
J(t)^2\leq \hat C^2\left(C_{3,\delta}+2C_{4,\delta} \lvert t \rvert ^{\frac{2}{3}}\right) \lvert t \rvert ^{\frac{2}{3}}.
\end{equation}
Hence, $C_J>0$ exists.

To show the lower of $J(t)$. From (\ref{eqn:1+1+1-body}),  (\ref{eqn:sundman}) and Theorem~\ref{thm:main 3.1}, for $t \in (-\delta,0)$,  we have
\begin{equation*}
\begin{aligned}
\dot J(t)=z(t) \times \ddot{z}(t)
&=\frac{m z(t)\times q(t)}{ \lvert z(t)-q(t) \rvert ^3}=\frac{m a(t) \lvert q(t) \rvert ^2\sin\theta(t)}{(a^2(t)+1+2a(t)\cos\theta(t))^{\frac{3}{2}} \lvert q(t) \rvert ^3}\\
&\geq \frac{m a(t)\sin\theta(t)}{(a(t)+1)^3C_q} \lvert t \rvert ^{-\frac{2}{3}}
\geq \frac{m c_a\sin\theta(t)}{(C_a+1)^3C_q} \lvert t \rvert ^{-\frac{2}{3}},
\end{aligned}
\end{equation*}
which is always positive. Then due to the assumption $\theta_*^-\in(0,\pi)$,  there exists $\epsilon\in(0,\delta)$ sufficiently small such that,  for $t \in (-\epsilon,0)$,
\begin{align*}
\dot J(t) \geq \check{C} \lvert t \rvert ^{-\frac{2}{3}},  \quad
\text{ where } \  \check{C}=\frac{m c_a\min\{\sin\theta_*^-,\sin\theta(-\epsilon)\}}{2(C_a+1)^3C_q}>0.
\end{align*}
Moreover, we observe that $J(t)<0$ since $\theta_*^-\in(0,\pi)$ and $\dot \theta(t)<0$ by Lemma \ref{lem:monotonity=0-}.
Then, for $t\in(-\epsilon,0)$,
\begin{equation*}
 \lvert J(t) \rvert
=-J(t)
=-J(0)+\int^0_t\dot J(s)ds
\geq \int^0_t\check{C} \lvert s \rvert ^{-\frac{2}{3}}ds
= 3\check{C} \lvert t \rvert ^{\frac{1}{3}},
\end{equation*}
where $J(0)=0$ by (\ref{eqn:estimate of J 1}). Hence, $c_{J}>0$ exists and the lemma holds.
\end{proof}
Now we are ready to prove Theorem \ref{thm:main 3.2}.

\begin{proof}[Proof of Theorem~\ref{thm:main 3.2}]
By contradiction, we assume $\theta_*^-\in(0,\pi)$. Since $z(0)=0$, by using Lemma \ref{lem: estimates of J} and Theorem~\ref{thm:main 3.1}, we see that
\begin{align}\label{eqn:dotangle}
c_JC_q^{-2}|t|^{-1} <  |\dot \theta(t)| <  C_Jc_q^{-2}|t|^{-1},\ \forall t\in (-\epsilon,0).
\end{align}
Combining \eqref{eqn:dotangle} with the monotonicity of $\theta(t)$, see Lemma \ref{lem:monotonity=0-}, we conclude that $$\theta(-\epsilon)-\theta(0)=\int^0_{-\epsilon}-\dot \theta(t)dt=\int^0_{-\epsilon}|\dot \theta(t)|dt=\infty.$$
This provides a contradiction, since the left hand side is finite. Therefore, $\theta_*^-\in\{0,\pi\}$. Moreover, if $\theta(t)\not\equiv \pi$, by Lemma \ref{lem:monotonity=0-} and the conclusion above, $\theta_*^-=0$ must happens. Hence, the proof is done.
\end{proof}

\subsection{Sundman-Sperling estimates}\label{subsec:sundman}

In this section, we aim to prove Theorem \ref{main1} completely, which involves the Sundman-Sperling estimates near the three-body collision in the restricted one-center-two-body problem.
For general restricted multi-body problem, once there exists a three-body collision, which is composed of a fixed particle, a moving primary and a massless particle, we believe the minimizer will also satisfy the Sundman-Sperling estimates near the three-body collision, since the impact of non-colliding particles is negligible.



The main steps of this proof is as follows:
\begin{itemize}
\item We first prove Lemma \ref{lem:limita*}, i.e. the limit $a^*=\lim_{t\rightarrow 0^-}a(t)$ exists, by using Theorem \ref{thm:main 3.1} and Theorem \ref{thm:main 3.2}. The technique of inflection and critical points is needed.

\item By using Lemma \ref{lem: zeros of h(a,theta)}, together with the technique of inflection and critical points, we prove that $a^*\in\{\alpha_1,\alpha_2,\alpha_3\}$, see Lemma \ref{lem:a*=alpha2}, \ref{lem:a*=alpha13}.

\item To complete the proof of Theorem \ref{main1}, it is sufficient to show that both $b(t)=\dot r(t)/\frac{d}{dt}|q(t)|$ and $c(t)=\ddot r(t)/\frac{d^2}{dt^2}|q(t)|$ are converging into $\{\alpha_1,\alpha_2,\alpha_3\}$ as $t\rightarrow 0^-$. We prove the former case in Lemma \ref{lem:b*=alpha123} and left the latter in the proof of Theorem \ref{main1}.
\end{itemize}


Recall that $z(t)=r(t)e^{i\theta(t)}$ is the action minimizer of $\mathcal{A}_{-T,0}$ on $\Omega^{-T,0}_{A,A_0}$. By assumption, we have $z(0)=0$, $z(t)$ is smooth on $[-t_0,0)$ and $z(-t_0)\in \mathbb{C}^+_*$ for some $t_0\in(0,T]$. Based on the previous results, we have the following lemma.

\begin{Lem}\label{lem:limita*}
$a^*:=\lim_{t\rightarrow 0^-}a(t)\in(0,+\infty)$.
\end{Lem}
\begin{proof}
According to Theorem \ref{thm:main 3.2}, the proof will be divided into two cases: $\theta_*^-=0$ and $\theta_*^-=\pi$. We first denote that $\underline{a}=\liminf_{t \rightarrow 0^-}a(t)$ and $\overline{a}=\limsup_{t \rightarrow 0^-}a(t)$. By Theorem~\ref{thm:main 3.1}, we see that  $c_a< \underline{a}\leq \overline{a} < C_a$ for some $c_a, C_a \in (0,+\infty)$.

By contradiction, we assume $\underline{a} < \overline{a}$.

Case 1: Assume $\theta_*^-=0$.

If $\underline{a}<\alpha_2$, on the one hand, there exists a  sequence  of  moments~$\{t_k\}$ with $t_k \rightarrow 0^-$    such  that
\begin{align}\label{eqn:assume1}
a(t_k)< (\underline{a}+\alpha_2)/2,\ \ \dot{a}(t_k)=0 \ \text{ and } \  \ddot{a}(t_k) \geq 0.
\end{align}
By $(\ref{eqn:dot a and ddot a})$ and Lemma~\ref{lem: ddot a(t) and h(a,theta)}, $(\ref{eqn:assume1})$ implies that $\dot{r}(t_k)=a(t_k)\frac{d}{dt}|q(t_k)|$ and
\begin{equation}\label{eqn:h(tk)>0}
h_{m/\mu}(a(t_k),\theta(t_k))+\frac{\tan^2\eta(t_k)}{\mu}a^3(t_k)\left(\frac{d}{dt} \lvert q(t_k) \rvert \right)^2 \lvert q(t_k) \rvert \geq0,
\end{equation}
where $\eta(t)$  denotes  the  angle  from  $z(t)$  to  $\dot{z}(t)$.

On the other hand, up to a subsequence, we assume the limit $ \lim_{k\rightarrow} a(t_k)=\hat a$ exists. Then, since  $\hat a\leq(\underline{a}+\alpha_2)/2<\alpha_2$, together with Lemma \ref{lem: zeros of h(a,theta)}(c)  and $\theta_*^-=0$, we have
\begin{align*}
\lim_{k\rightarrow +\infty}h_{m/\mu}(a(t_k),\theta(t_k)) = h_{m/\mu}(\hat a,0)<-\epsilon,
\end{align*}
for some $\epsilon>0$ sufficiently small. Moreover, by Theorem~\ref{thm:main 3.1} and $(\ref{eqn:sundman})$, we conclude that both $a(t_k)$ and
$\left(\frac{d}{dt} \lvert q(t_k) \rvert \right)^2 \lvert q(t_k) \rvert \leq C_q^2 \lvert t_k \rvert ^{-\frac{2}{3}}C_q \lvert t_k \rvert ^{\frac{2}{3}}=C_q^3$ are bounded. As $k\rightarrow +\infty$, by using Lemma~\ref{lem:dot theta d-}, that is  $\tan^2\eta(t_k)\rightarrow0$, we conclude that the left side of (\ref{eqn:h(tk)>0}) is strictly negative when $k$ is large enough, which gives a contradiction.

If $\underline{a}\geq \alpha_2$. Then we see that $\bar a>\alpha_2$. Analog to the situation before, there exists a sequence $\{s_k\}$ with $s_k \rightarrow 0^-$  such that
\begin{align}\label{eqn:assume2}
a(s_k) > (\underline{a}+\bar a)/2,\ \ \dot{a}(s_k)=0 \ \text{ and } \  \ddot{a}(s_k) \leq 0.
\end{align}
By $(\ref{eqn:dot a and ddot a})$ and Lemma~\ref{lem: ddot a(t) and h(a,theta)} again,
$(\ref{eqn:assume2})$ implies that $\dot{r}(s_k)=a(s_k)\frac{d}{dt}|q(s_k)|$ and
\begin{equation}\label{eqn:h(tk)<0}
h_{m/\mu}(a(s_k),\theta(s_k))+\frac{\tan^2\eta(s_k)}{\mu}a^3(s_k)\left(\frac{d}{dt} \lvert q(s_k) \rvert \right)^2 \lvert q(s_k) \rvert \leq0.
\end{equation}
Up to a subsequence, we assume the limit $\lim_{k\rightarrow +\infty}a(s_k)=\check a$ exists.
Then, since $\check a\geq(\underline a+ \overline a)/2>\alpha_2$, same as the situation before, we can conclude that the left side of (\ref{eqn:h(tk)<0}) is strictly positive when $k$ is large enough,  which leads to another contradiction.
Hence, the limit $a^*$ exists. The proof of Case 1 is completed.

Case 2: Assume $\theta_*^-=\pi$. By Lemma \ref{lem:monotonity=0-}, the solution satisfies $z(t)\in (-\infty,q(t))$ or $z(t)\in (q(t),0)$ for any $t\in[-t_0,0)$. For the former case, we see that $a(t)>1$ for any $t\in[-t_0,0)$. This implies that $\overline a>\underline a\geq 1$. Same as in Case 1, when $\underline a\in[1,\alpha_3)$, one can find a sequence $t_k$ such that \eqref{eqn:h(tk)>0} holds for any $k$. However, as $k$ sufficiently large, one can choose a subsequence, also denote by $t_k$, such that $a(t_k)\rightarrow \hat a\in [1^+,\alpha_3)$. Since Lemma \ref{lem: zeros of h(a,theta)}, we see that $h_{m/\mu}(a,\pi)<0$ for any $a\in(1,\alpha_3)$ and $h_{m/\mu}(a,\pi)>0$ for any $a\in(\alpha_3,+\infty)$. In particular, as $a\rightarrow 1^+$, $h_{m/\mu}(a,\pi)\rightarrow -\infty$. Therefore, for sufficiently large $k$, \eqref{eqn:h(tk)>0} become negative, which is a contradiction. Moreover, when $\underline a\in(\alpha_3,+\infty)$, by following the previous argument, we also obtain a contradiction. Hence, the limit $a^*$ exists. Similarly, by using the same strategy for $\underline a\in(0,\alpha_1)$ and $\underline a\in (\alpha_1,1)$, the latter case also lead to a contradiction. The proof of Case 2 is completed.


\end{proof}

In order to characterize the limit $a^*$, we first provide the following asymptotic estimates of $\dot{r}$.

\begin{Lem}\label{lem:rangedotr}
Assume $\theta_*^-\in\{0,\pi\}$. There exist $\epsilon>0$ and $0<c_d<C_d<\infty$ such that
\begin{equation*}
c_d |t|^{\frac{1}{3}}  <   |\dot{r}(t)|   <  C_d|t|^{\frac{1}{3}}, \quad \forall\  t\in(-\epsilon,0).
\end{equation*}
\end{Lem}

\begin{proof}
Firstly, we introduce the following useful identity
 \begin{align}\label{eqn:ddotequation2}
\ddot{r}(t) = \left( -\frac{\mu}{a^2(t)}-\frac{m(a(t)+\cos\theta(t))}{(a^2(t)+1+2a(t)\cos\theta(t))^{\frac{3}{2}}}\right)|q(t)|^{-2}+\tan^2\eta(t)\frac{\dot{r}^2(t)}{r(t)},
\end{align}
where $\eta(t)$ denotes the angle from $z(t)$ to $\dot{z}(t)$. The computation can be found in (\ref{eqn:ddotrequation}) and (\ref{eqn:ddotrh}). We split the proof in two cases.



Case 1: Assume $\theta_*^-=0$.

By Lemma~\ref{lem:limita*}, for any $\epsilon>0$, we see that
\begin{align}\label{eqn:rangeepsilon}
\left|\left(-\frac{\mu}{a^2(t)}-\frac{m(a(t)+\cos\theta(t))}{a^2(t)+1+2a(t)\cos\theta(t)^{\frac{3}{2}}}\right) - \left(-\frac{\mu}{(a^*)^2}-\frac{m}{(a^*+1)^2}\right) \right| <\epsilon,
\end{align}
for $|t|>0$ sufficiently small.

On the one hand, fixing a $t_0<0$ near $0$. By using (\ref{eqn:sundman}), (\ref{eqn:ddotequation2}) and (\ref{eqn:rangeepsilon}),  we have
\begin{align*}
\dot{r}(t) - \dot{r}(t_0)
 = \int_{t_0}^t \ddot{r}(s) ds
& \geq \int_{t_0}^t \left( -\frac{\mu}{a^2(s)}-\frac{m(a(s)+\cos\theta(s))}{(a^2(s)+1+2a(s)\cos\theta(s))^{\frac{3}{2}}}\right)|q(s)|^{-2} ds \\
& \geq \int_{t_0}^t \left(-\frac{\mu}{(a^*)^2}-\frac{m}{(a^*+1)^2} -\epsilon\right) c_q^{-2}|s|^{-\frac{4}{3}} ds \\
& = -C'_d\left( |t|^{-\frac{1}{3}}-|t_0|^{-\frac{1}{3}}\right),
\end{align*}
where $C'_d:= 3c_q^{-2}\left(\frac{\mu}{(a^*)^2}+\frac{m}{(a^*+1)^2} +\epsilon\right)>0$.  Then, by slightly enlarging $C'_d$, we derive the lower bound for $\dot{r}$: for some $C_d>C'_d$,
\begin{align}\label{eqn:dotrlower1}
\dot{r}(t) \geq -C'_d|t|^{-\frac{1}{3}}+C'_d|t_0|^{-\frac{1}{3}}+\dot{r}(t_0) > -C_d|t|^{-\frac{1}{3}},
\end{align}
for $|t|$ sufficiently small.

On the other hand,  due to inequalities (\ref{eqn:sundman}),  (\ref{eqn:dotrlower1}), Lemma~\ref{lem:dot theta d-} and Lemma~\ref{lem:limita*}, as~$t$~tends to $0^-$,   $\dot{r}^2(t)|q(t)|a^{-1}(t)$ is bounded and $\tan^2\eta(t)$ tends to 0.
Then, fixing a $t_0<0$ near 0 again,
by a similar computation as above, we have
\begin{align*}
\dot{r}(t) - \dot{r}(t_0)
& = \int_{t_0}^t \left( -\frac{\mu}{a^2(s)}-\frac{m(a(s)+\cos\theta(s))}{(a^2(s)+1+2a(s)\cos\theta(s))^{\frac{3}{2}}}+\tan^2\eta(s)\frac{\dot{r}^2(s)}{a(s)}|q(s)| \right)|q(s)|^{-2} \ ds \\
& \leq \int_{t_0}^t \left(-\frac{\mu}{(a^*)^2}-\frac{m}{(a^*+1)^2} +2\epsilon\right) C_q^{-2}|s|^{-\frac{4}{3}} ds \\
& = -c'_d\left( |t|^{-\frac{1}{3}}-|t_0|^{-\frac{1}{3}}\right),
\end{align*}
 where $c'_d:= 3C_q^{-2}\left(\frac{\mu}{(a^*)^2}+\frac{m}{(a^*+1)^2} -2\epsilon\right)>0$.
Finally, after reducing $c'_d$ slightly,  we derive the upper bound of $\dot{r}$: for some $c_d<c'_d$,
\begin{align*}
\dot{r}(t) \leq -c'_d|t|^{-\frac{1}{3}}+c'_d|t_0|^{-\frac{1}{3}}+\dot{r}(t_0)
< -c_d|t|^{-\frac{1}{3}},
\end{align*}
for $|t|$ sufficiently small.  This proves the lemma in Case 1.

Case 2. Assume $\theta_*^-=\pi$.

By Lemma \ref{lem:monotonity=0-}, we know that $\theta(t)\equiv \pi$ for any $t\in(-t_0,0)$. Since the angle $\eta(t)$ between $z(t)$ and $\dot z(t)$ is either $0$ or $\pi$ if $\dot r\neq0$, then $\tan\eta(t)\frac{\dot r(t)}{r(t)}\equiv 0$ for every $t\in(-t_0,0)$. By Lemma \ref{lem:limita*}, for any $\epsilon>0$, we have
\begin{align}\label{eqn:rangeepsilon1}
\left|\left(-\frac{\mu}{a^2(t)}-\frac{m(a(t)+\cos\theta(t))}{a^2(t)+1+2a(t)\cos\theta(t)^{\frac{3}{2}}}\right) - \left(-\frac{\mu}{(a^*)^2}-\frac{m}{(a^*-1)^2}\right) \right| <\epsilon,
\end{align}
for $|t|$ sufficiently small. Following the previous strategy. Fixing a $t_0<0$ near 0. By using (\ref{eqn:sundman}), (\ref{eqn:ddotequation2}) and (\ref{eqn:rangeepsilon1}), we have
$$\begin{aligned}
\dot{r}(t) - \dot{r}(t_0)
 = \int_{t_0}^t \ddot{r}(s) ds
& = \int_{t_0}^t \left( -\frac{\mu}{a^2(s)}-\frac{m(a(s)+\cos\theta(s))}{(a^2(s)+1+2a(s)\cos\theta(s))^{\frac{3}{2}}}\right)|q(s)|^{-2} ds \\
& \geq \int_{t_0}^t \left(-\frac{\mu}{(a^*)^2}-\frac{m}{(a^*-1)^2} -\epsilon\right) c_q^{-2}|s|^{-\frac{4}{3}} ds \\
& = -C''_d\left( |t|^{-\frac{1}{3}}-|t_0|^{-\frac{1}{3}}\right),
\end{aligned}$$
where $C''_d:= 3c_q^{-2}\left(\frac{\mu}{(a^*)^2}+\frac{m}{(a^*-1)^2} +\epsilon\right)>0$.
Moreover, same as before, we also obtain that
\begin{align*}
\dot{r}(t) - \dot{r}(t_0)
& = \int_{t_0}^t \left( -\frac{\mu}{a^2(s)}-\frac{m(a(s)+\cos\theta(s))}{(a^2(s)+1+2a(s)\cos\theta(s))^{\frac{3}{2}}}
\right)|q(s)|^{-2} \ ds \\
& \leq \int_{t_0}^t \left(-\frac{\mu}{(a^*)^2}-\frac{m}{(a^*-1)^2} +\epsilon\right) C_q^{-2}|s|^{-\frac{4}{3}} ds \\
& = -c''_d\left( |t|^{-\frac{1}{3}}-|t_0|^{-\frac{1}{3}}\right),
\end{align*}
 where $c''_d:= 3C_q^{-2}\left(\frac{\mu}{(a^*)^2}+\frac{m}{(a^*-1)^2} -\epsilon\right)>0$. By choosing $\epsilon>0$ sufficiently small, both $C''_d$ and $c''_d$ are positive. Then by enlarging $C''_d$ and reducing $c''_d$ as before, we obtain the lower and upper bound for $\dot r$, no matter $\alpha^*=\alpha_1$ or $\alpha_3$. This proves the lemma in Case 2 and the proof is now completed.
\end{proof}

Now we introduce the following lemma, i.e. when the limiting angle $\theta_*^-=0$, then $r$ and $|q|$ tend to be proportional near the three-body collision. Moreover, the limiting ratio can be determined only by the mass ratio $m/\mu$.

\begin{Lem}\label{lem:a*=alpha2}
Assume $\theta_*^-=0$, then $a^* = \alpha_2$, where $\alpha_2$ is the unique zero of \eqref{eqn:h(a,0)} in $(1,+\infty)$.
\end{Lem}

\begin{proof}


To show $a^*\leq \alpha_2$.  Assume not, that is $a^*>\alpha_2$.
Let $\epsilon >0$ and $a_{\epsilon}=a^*-\epsilon$. Since $a(t)$ is convergent at $t=0$, see Lemma \ref{lem:limita*}, we know that $a(t)-a_{\epsilon}>0$ for $|t|>0$ sufficiently small. 
Moreover, combining (\ref{eqn:sundman}), Lemma~\ref{lem:dot theta d-}, Lemma~\ref{lem: zeros of h(a,theta)}(c) and Lemma ~\ref{lem:limita*}, \ref{lem:rangedotr}, we have
\begin{align*}
\lim_{t \rightarrow 0^-}h_{m/\mu}(a(t),\theta(t))=h_{m/\mu}(a^*,0)>0 \quad \text{ and } \quad
 \lim_{t \rightarrow 0^-}\tan^2\eta(t)\frac{\dot{r}^2(t)}{a(t)}|q(t)|=0.
\end{align*}

Then by (\ref{eqn:ddotrh}) and $\frac{d^2}{dt^2}|q(t)|=-\frac{\mu}{|q(t)|^2}$, we compute that, for $|t|>0$ sufficiently small,
\begin{equation}\label{eqn:ddotrh2}
\begin{aligned}
\ddot{r}(t)-a_{\epsilon}\frac{d^2}{dt^2}|q(t)|
&= \left(\ddot{r}(t)-a(t)\frac{d^2}{dt^2}|q(t)|\right)+\left(a(t)-a_{\epsilon}\right)\frac{d^2}{dt^2}|q(t)|  \\
&=\left(h_{m/\mu}(a(t),\theta(t))+\frac{\tan^2\eta(t)}{\mu}\dot r^2(t)r(t)-a^2(t)(a(t)-a_{\epsilon})\right)\frac{\mu}{r^2(t)} \\
& \geq \left(h_{m/\mu}(a^*,0)-\epsilon-a^2(t)\cdot2\epsilon\right)\frac{\mu}{r^2(t)} \\
& > \left(\frac{1}{2(a^*)^2}h_{m/\mu}(a^*,0)-\frac{2\epsilon}{(a^*)^2}-2\epsilon\right)\frac{\mu}{|q(t)|^2}
=c_{\epsilon}|q(t)|^{-2},
\end{aligned}
\end{equation}
where $c_{\epsilon}:= \left(\frac{1}{2(a^*)^2}h_{m/\mu}(a^*,0)-\frac{2\epsilon}{(a^*)^2}-2\epsilon\right)\mu.$
Since $\epsilon>0$ is arbitrary,  we choose $\epsilon>0$ sufficiently small such that
$c_{\epsilon} >0$.

Fix a $t_0<0$ near 0. By integration and Lemma~\ref{lem:sundman},  we have,   for $t \in (t_0,0)$,
\begin{align*}
\dot{r}(t)-a_{\epsilon}\frac{d}{dt}|q(t)|
& = c_0 +\int_{t_0}^{t}\ddot{r}(s)-a_{\epsilon}\frac{d^2}{dt^2}|q(s)| \ ds
 \geq c_0+\int_{t_0}^{t} c_{\epsilon}|q(s)|^{-2} \ ds \\
& \geq c_0 + \frac{3c_{\epsilon}}{C_q^2}\left( |t|^{-\frac{1}{3}}-|t_0|^{-\frac{1}{3}}\right)
=   \frac{3c_{\epsilon}}{C_q^2}|t|^{-\frac{1}{3}} +c_1,
\end{align*}
where $c_0:=\dot{r}(t_0)-a_{\epsilon}\frac{d}{dt}|q(t_0)| $ and $c_1:=c_0- \frac{3c_{\epsilon}}{C_q^2}|t_0|^{-\frac{1}{3}}$.
By integration again,  we obtain
\begin{align*}
0 & > -(a(t)-a_{\epsilon})|q(t)| =-(r(t)-a_{\epsilon}|q(t)|)  \\
& =\int_{t}^{0} \dot{r}(s)-a_{\epsilon}\frac{d}{dt}|q(s)| \ ds
 \geq \int_{t}^{0}  \frac{3c_{\epsilon}}{C_q^2}|s|^{-\frac{1}{3}} +c_1 \ ds
  =  \left(\frac{9c_{\epsilon}}{2C_q^2} +c_1|t|^{\frac{1}{3}}\right)|t|^{\frac{2}{3}} > 0,
\end{align*}
for $|t|>0$ sufficiently small.  Then, we obtain a contradiction and we conclude that $a^*\leq \alpha_2$.

To prove $a^*\geq \alpha_2$. 
By contradiction, we assume $a^* < \alpha_2$. Let $\epsilon >0$ and $a_{\epsilon}=a^*+\epsilon$. By Lemma \ref{lem:limita*} again, we have $a(t)-a_{\epsilon}\in (-2\epsilon,0)$ for $|t|>0$ sufficiently small. Similar to the previous case, we have
\begin{align}\label{eqn:limithtan}
\lim_{t \rightarrow 0^-}h_{m/\mu}(a(t),\theta(t))=h_{m/\mu}(a^*,0)<0 \quad \text{ and } \quad
 \lim_{t \rightarrow 0^-}\tan^2\eta(t)\frac{\dot{r}^2(t)}{a(t)}|q(t)|  =0,
\end{align}

Then similar to the estimates in (\ref{eqn:ddotrh2}), for $|t|>0$ sufficiently small, we compute that
\begin{equation}\label{eqn:ddotrh3}
\begin{aligned}
\ddot{r}(t)-a_{\epsilon}\frac{d^2}{dt^2}|q(t)|
&=\left(h_{m/\mu}(a(t),\theta(t))+\frac{\tan^2\eta(t)}{\mu}\dot r^2(t)r(t)-a^2(t)(a(t)-a_{\epsilon})\right)\frac{\mu}{r^2(t)} \\
& \leq \left(h_{m/\mu}(a^*,0)+2\epsilon+a^2(t)\cdot2\epsilon\right)\frac{\mu}{r^2(t)} \\
& < \left(\frac{1}{2(a^*)^2}h_{m/\mu}(a^*,0)+\frac{3\epsilon}{(a^*)^2}+2\epsilon\right)\frac{\mu}{|q(t)|^2}
=-C_{\epsilon}|q(t)|^{-2},
\end{aligned}
\end{equation}
where $C_{\epsilon}:=-\left(\frac{1}{2(a^*)^2}h_{m/\mu}(a^*,0)+\frac{3\epsilon}{(a^*)^2}+2\epsilon\right)\mu$.  Since $\epsilon>0$ is also arbitrary, we can choose $\epsilon>0$ sufficiently small such that $C_{\epsilon}>0$.

Next, we fix a $t_0<0$ near 0. By integration and Lemma~\ref{lem:sundman}, for any $t \in (t_0,0)$, we have
\begin{align*}
\dot{r}(t)-a_{\epsilon}\frac{d}{dt}|q(t)|
& = C_0 +\int_{t_0}^{t}\ddot{r}(s)-a_{\epsilon}\frac{d^2}{dt^2}|q(s)| \ ds
 \leq C_0-\int_{t_0}^{t} C_{\epsilon}|q(s)|^{-2} \ ds \\
& \leq C_0 - \frac{3C_{\epsilon}}{C_q^2}\left( |t|^{-\frac{1}{3}}-|t_0|^{-\frac{1}{3}}\right)
=   -\frac{3C_{\epsilon}}{C_q^2}|t|^{-\frac{1}{3}} +C_1,
\end{align*}
where $C_0:=\dot{r}(t_0)-a_{\epsilon}\frac{d}{dt}|q(t_0)| $ and $C_1:=C_0+ \frac{3C_{\epsilon}}{C_q^2}|t_0|^{-\frac{1}{3}}$.
By integration again,  we obtain
\begin{align*}
0 & < -(a(t)-a_{\epsilon})|q(t)| =-(r(t)-a_{\epsilon}|q(t)|)  \\
& =\int_{t}^{0} \dot{r}(s)-a_{\epsilon}\frac{d}{dt}|q(s)| \ ds
 \leq \int_{t}^{0}  -\frac{3C_{\epsilon}}{C_q^2}|s|^{-\frac{1}{3}} +C_1 \ ds
  =  \left( - \frac{9C_{\epsilon}}{2C_q^2} +C_1|t|^{\frac{1}{3}}\right)|t|^{\frac{2}{3}} < 0,
\end{align*}
for $|t|>0$ sufficiently small.  Then, we obtain a contradiction and prove that $a^*\geq \alpha_2$. The proof is now completed.
\end{proof}

Near the three-body collision, the ratio of $|z|$ and $|q|$ has been characterized in the case $\theta_*^-=0$. Now we consider the case $\theta_*^-=\pi$.

\begin{Lem}\label{lem:a*=alpha13}
Assume $\theta_*^-=\pi$, then $a^*\in\{\alpha_1,\alpha_3\}$, where $\alpha_1<1<\alpha_3$ are the unique two zeros of \eqref{eqn:h(a,pi)} in $(0,1)\cup(1,+\infty)$.
\end{Lem}
\begin{proof}
Since $\theta(t)\equiv \pi$ on $(-t_0,0)$, there are only two possibilities, i.e. $z(t)\in(-\infty,q(t))$ and $z(t)\in(q(t),0)$. We aim to show that, $a^*=\alpha_1$ for the form case, and $a^*=\alpha_3$ for the latter case. Since the proof is similar to Lemma \ref{lem:a*=alpha2}, then we only sketch the proof for $a^*=\alpha_1$ and omit the proof for $a^*=\alpha_3$.

We split the proof in two situations: to show the contradictions for $a^*\in(0,\alpha_1)$ and $a^*\in(\alpha_1,1]$.

Assume $a^*\in(\alpha_1,1]$. Let $\epsilon >0$ and $a_{\epsilon}=a^*-\epsilon$. By Lemma \ref{lem:limita*}, we know that $a(t)-a_{\epsilon}>0$ for $|t|>0$ sufficiently small.
Moreover, since $\theta(t)\equiv \pi$, the angle $\eta(t)$ from $z(t)$ to $\dot z(t)$ is $0$ or $\pi$ if $\dot r(t)\neq0$. Then by Lemma~\ref{lem: zeros of h(a,theta)}(c), we have
\begin{align*}
\lim_{t \rightarrow 0^-}h_{m/\mu}(a(t),\theta(t))=h_{m/\mu}(a^*,0)>0 \quad \text{ and } \quad
 \tan^2(\eta(t))\dot{r}^2(t)r(t)=0,\quad \forall t\in(-t_0,0).
\end{align*}
By the computation in (\ref{eqn:ddotrh2}), for $|t|>0$ sufficiently small, we have
\begin{equation*}
\begin{aligned}
\ddot{r}(t)-a_{\epsilon}\frac{d^2}{dt^2}|q(t)|
&=\left(h_{m/\mu}(a(t),\theta(t))-a^2(t)(a(t)-a_{\epsilon})\right)\frac{\mu}{r^2(t)} \\
& > \left(\frac{1}{2(a^*)^2}h_{m/\mu}(a^*,0)-\frac{2\epsilon}{(a^*)^2}-2\epsilon\right)\frac{\mu}{|q(t)|^2}
=c_{\epsilon}|q(t)|^{-2},
\end{aligned}
\end{equation*}
where $c_{\epsilon}:= \left(\frac{1}{2(a^*)^2}h_{m/\mu}(a^*,0)-\frac{2\epsilon}{(a^*)^2}-2\epsilon\right)\mu.$
Since $\epsilon>0$ is arbitrary,  we choose $\epsilon>0$ sufficiently small such that
$c_{\epsilon} >0$. On the other hand, fix a $t_0<0$ near $0$. By following exactly the argument in Lemma~\ref{lem:a*=alpha2}, we obtain a contradiction that $0>-(a(t)-a_\epsilon)|q(t)|>0$ and we have $a^*\leq \alpha_1$.

Assume $a^*\in(0,\alpha_1)$. Let $\epsilon >0$ and $a_{\epsilon}=a^*+\epsilon$. By Lemma~\ref{lem:limita*} again, we have $a(t)-a_{\epsilon}\in (-2\epsilon,0)$ for $|t|>0$ sufficiently small. Similar to the previous case, we have
\begin{align*}
\lim_{t \rightarrow 0^-}h_{m/\mu}(a(t),\theta(t))=h_{m/\mu}(a^*,0)<0 \quad \text{ and } \quad
 \tan^2\eta(t)\dot{r}^2(t)r(t)=0,
\end{align*}
Similar to the estimates \eqref{eqn:ddotrh3}, for $|t|>0$ sufficiently small, we compute that
\begin{equation*}
\begin{aligned}
\ddot{r}(t)-a_{\epsilon}\frac{d^2}{dt^2}|q(t)|
&=\left(h_{m/\mu}(a(t),\theta(t))-a^2(t)(a(t)-a_{\epsilon})\right)\frac{\mu}{r^2(t)} \\
& < \left(\frac{1}{2(a^*)^2}h_{m/\mu}(a^*,0)+\frac{3\epsilon}{(a^*)^2}+2\epsilon\right)\frac{\mu}{|q(t)|^2}
=-C_{\epsilon}|q(t)|^{-2},
\end{aligned}
\end{equation*}
where $C_{\epsilon}:=-\left(\frac{1}{2(a^*)^2}h_{m/\mu}(a^*,0)+\frac{3\epsilon}{(a^*)^2}+2\epsilon\right)\mu$.  Since $\epsilon>0$ is also arbitrary, we can choose $\epsilon>0$ sufficiently small such that $C_{\epsilon}>0$. Next, we fix a $t_0<0$ near 0. By using the same strategy as in Lemma \ref{lem:a*=alpha13}, we obtain a similar contradiction, $0<-(a(t)-a_\epsilon)|q(t)|<0$. This means $a^*\geq \alpha_1$. The proof is now completed.
\end{proof}

To further prove the Sundman-Sperling estimates for the three-body collision, we study the ratio of velocities for $z$ and $q$, based on the characterization of the ratio of positions.

\begin{Lem}\label{lem:b*=alpha123}
Let $b(t)=\dot{r}(t)/\frac{d}{dt}|q(t)|$. Then we have $b^*:=\lim_{t\rightarrow 0^-}b(t)$ exists. Moreover,
\begin{enumerate}
\item[$(a)$] if $\theta_*^-=0$, then $b^*= \alpha_2$,
\item[$(b)$] if $\theta_*^-=\pi$, then $b^*\in\{\alpha_1,\alpha_3\}$,
\end{enumerate}
where $\alpha_2$ solves \eqref{eqn:h(a,0)} and $\alpha_1,\alpha_3$ solve \eqref{eqn:h(a,pi)}.
\end{Lem}

\begin{proof}
To show (a). Following the strategy of Lemma \ref{lem:a*=alpha2}. We first prove the existence of $b^*$, then prove $b^*=\alpha_2$. Denote $\underline{b}^-=\liminf_{t\rightarrow 0^-}b(t)$ and $\overline{b}^-=\limsup_{t\rightarrow 0^-}b(t)$. Unlike behavior of $a(t)$, $\underline{b}^-$ and $\overline{b}^-$ might be $\pm \infty$.

To show $b^*$ exists. By contradiction, we assume $\underline{b}^-<\overline{b}^-$.

Case 1: Assume $\underline{b}^-<\alpha_2$. There exists an $\epsilon>0$ small and a sequence of moments $\{t_k\}$ with $t_k \rightarrow 0^-$ such that
\begin{align}\label{eqn:limitb*}
b(t_k)<\alpha_2-\epsilon,\quad \text{ and } \quad \dot{b}(t_k)= \frac{\ddot{r}(t_k)-b(t_k)\frac{d^2}{dt^2}|q(t_k)|}{\frac{d}{dt}|q(t_k)|} \leq 0,
\end{align}
for any $k$ sufficiently large.

From Lemma~\ref{lem:a*=alpha2},  it is clear that $a(t_k)-b(t_k)\geq\epsilon/2$ for sufficiently large $k$. By similar arguments in (\ref{eqn:limithtan}) and (\ref{eqn:ddotrh3}), together with the fact that $h_{m/\mu}(\alpha_2,0)=0$, we conclude that \begin{equation}\label{eqn:ddotrh6}
\begin{aligned}
\ddot{r}(t_k)&-b(t_k)\frac{d^2}{dt^2}|q(t_k)| \\
&=\left(h_{m/\mu}(a(t_k),\theta(t_k))+\frac{\tan^2\eta(t_k)}{\mu}\dot r^2(t_k)r(t_k)-a^2(t_k)(a(t_k)-b(t_k))\right)\frac{\mu}{r^2(t_k)} \\
& \leq \left(h_{m/\mu}(\alpha_2,0)-\frac{\epsilon}{4}a^2(t_k)\right)\frac{\mu}{r^2(t_k)}<0,
\end{aligned}
\end{equation}
for $k$ sufficiently large. This contradicts to (\ref{eqn:limitb*}) since $\dot{b}(t_k) \leq 0$ and $\frac{d}{dt}|q(t_k)|<0$.

Case 2: Assume $\underline{b}^-\geq \alpha_2$. There also exists an $\epsilon>0$ and a sequence  $\{s_k\}$ with $s_k \rightarrow 0^-$  such  that
\begin{align}\label{eqn:limitb*2}
b(s_k)>\alpha_2+\epsilon,\quad \text{ and } \quad \dot{b}(s_k)= \frac{\ddot{r}(s_k)-b(s_k)\frac{d^2}{dt^2}|q(s_k)|}{\frac{d}{dt}|q(s_k)|} \geq 0,
\end{align}
for $k$ sufficiently large.

From Lemma~\ref{lem:a*=alpha2} again,  it is clear that $a(t_k)-b(t_k)\leq-\epsilon/2$ for sufficiently large $k$. By a similar argument in (\ref{eqn:ddotrh6}),  for $k$ sufficiently large, we have
\begin{equation*}
\begin{aligned}
\ddot{r}(t_k)-b(t_k)\frac{d^2}{dt^2}|q(t_k)|
\geq \left(h_{m/\mu}(\alpha_2,0)+\frac{\epsilon}{4}a^2(t_k)\right)\frac{\mu}{r^2(t_k)}>0,
\end{aligned}
\end{equation*}
which leads to a contradiction with (\ref{eqn:limitb*2}). Therefore, the limit $\underline{b}^-=\overline b^-=b^*$ exists.

To show $b^*=\alpha_2$. By contradiction, we assume $b^* \neq \alpha_2$.  Let $\epsilon = \frac{1}{2}|b^*-\alpha_2|>0$,  then for $|t|$ sufficiently small, we have
\begin{equation}\label{eqn:final}
\begin{aligned}
\left| a(t/2)-a(t)\right|
&= \left| \int_{t}^{t/2} \dot{a}(s)ds \right|
 = \left| \int_{t}^{t/2} \frac{\dot{r}(s)|q(s)|-r(s)\frac{d}{dt}|q(s)|}{|q(s)|^2} ds \right| \\
& = \left| \int_{t}^{t/2} \frac{(b(s)-a(s))\frac{d}{dt}|q(s)|}{|q(s)|} ds \right| \geq \int_{t}^{t/2}(|b^*-\alpha_2|-\epsilon)\frac{\frac{d}{dt}|q(s)|}{|q(s)|} ds \\
& \geq \frac{|b^*-\alpha_2|}{2} \int_{t}^{t/2}\frac{c_q}{C_q}|s|^{-1} ds
= \frac{c_q }{2C_q}  |b^*-\alpha_2|  \ln 2 >0.
\end{aligned}
\end{equation}
The first inequality above follows from the fact that $(b(s)-a(s))\frac{d}{dt}|q(s)|$
 does not change  the sign when  $|s|$ is sufficiently small,  and the second inequality follows from $(\ref{eqn:sundman})$.  However,  from the fact that $a(t)$ is convergent,  we observe that $\lim_{t \rightarrow 0^-}|a(t/2)-a(t)|=0$,  which contradicts (\ref{eqn:final}). The proof of (a) is completed.

To show (b). By Lemma \ref{lem:monotonity=0-}, we know that $\theta(t)\equiv \pi$ for any $t\in(-t_0,0)$. Since the angle $\eta(t)$ between $z(t)$ and $\dot z(t)$ is either $0$ or $\pi$ if $\dot r\neq0$, then $\tan\eta(t)\frac{\dot r(t)}{r(t)}\equiv 0$ for every $t\in(-t_0,0)$. Similar to Lemma \ref{lem:a*=alpha13}, we first prove that $b^*$ exists, then prove $b^*=\alpha_1$ or $b^*=\alpha_3$.

Since the proof of $b^*=\alpha_1$ and $b^*=\alpha_3$ are similar, we only sketch the proof for $b^*=\alpha_1$. By contradiction,
assume $\underline{b}^-<\overline b^-$, where $\underline{b}^-$ and $\overline{b}^-$ might be $\pm \infty$. Similar as (a), we will show that both $\underline{b}^-<\alpha_1$ and $\underline{b}^-\geq \alpha_1$ will lead to a contradiction.

Case 1: When $\underline{b}^-<\alpha_1$, there exists an $\epsilon>0$ small and a sequence of moments $\{t_k\}$ with $t_k \rightarrow 0^-$ such that, for $k$ sufficiently large, we have
\begin{align}\label{eqn:limitb*3}
b(t_k)<\alpha_1-\epsilon,\quad \text{ and } \quad \dot{b}(t_k)= \frac{\ddot{r}(t_k)-b(t_k)\frac{d^2}{dt^2}|q(t_k)|}{\frac{d}{dt}|q(t_k)|} \leq 0.
\end{align}
From Lemma~\ref{lem:a*=alpha13},  it is clear that $a(t_k)-b(t_k)>\epsilon/2$. Same as the proof of (a), together with the fact that $h_{m/\mu}(\alpha_1,0)=0$,  we have
\begin{equation}\label{eqn:ddotrh7}
\begin{aligned}
\ddot{r}(t_k)-b(t_k)\frac{d^2}{dt^2}|q(t_k)|
&=\left(h_{m/\mu}(a(t_k),\theta(t_k))-a^2(t_k)(a(t_k)-b(t_k))\right)\frac{\mu}{r^2(t_k)} \\
& \leq \left(h_{m/\mu}(\alpha_1,0)-\frac{\epsilon}{4}a^2(t_k)\right)\frac{\mu}{r^2(t)} <0,
\end{aligned}
\end{equation}
for $k$ sufficiently large. This contradicts to (\ref{eqn:limitb*3}) since $\dot{b}(t_k) \leq 0$ and $\frac{d}{dt}|q(t_k)|<0$.

Case 2: When $\underline{b}^-\geq\alpha_1$, there exists an $\epsilon>0$ small and a sequence $\{s_k\}$ with $s_k \rightarrow 0^-$  such that, for $k$ sufficiently large, we have
\begin{align}\label{eqn:limitb*4}
b(s_k)>\alpha_1+\epsilon,\quad \text{ and } \quad \dot{b}(s_k)= \frac{\ddot{r}(s_k)-b(s_k)\frac{d^2}{dt^2}|q(s_k)|}{\frac{d}{dt}|q(s_k)|} \geq 0.
\end{align}
This means that $a(s_k)-b(s_k)<-\epsilon/2$.
Similar to the argument in (\ref{eqn:ddotrh7}), we have
\begin{equation*}
\begin{aligned}
\ddot{r}(t_k)-b(t_k)\frac{d^2}{dt^2}|q(t_k)|
 \geq \left(h_{m/\mu}(\alpha_1,0)+\frac{\epsilon}{4}a^2(t_k)\right)\frac{\mu}{r^2(t)}>0,
\end{aligned}
\end{equation*}
for $k$ sufficiently large. This also leads to a contradiction with (\ref{eqn:limitb*4}).  Hence, the limit $\underline{b}^-=\overline b^-=b^*$ exists.

To show $b^*=\alpha_1$. By contradiction, we assume $b^* \neq \alpha_1$.  Let $\epsilon = \frac{1}{2}|b^*-\alpha_1|>0$,  then same as the computation in (a), we conclude that, for $|t|$ sufficiently small,
\begin{equation}\label{eqn:final1}
\begin{aligned}
\left| a(t/2)-a(t)\right|
= \frac{c_q }{2C_q}  |b^*-\alpha_1|  \ln 2 >0.
\end{aligned}
\end{equation}
However, $\lim_{t \rightarrow 0^-}|a(t/2)-a(t)|=0$, since $a(t)$ is convergent. This contradicts to (\ref{eqn:final1}). Therefore, (b) holds as we expected. The proof is now completed.
\end{proof}

Now we are ready to prove Theorem~\ref{main1}.
\begin{proof}[Proof of Theorem~\ref{main1}]
Recall that the restricted one-center-two-body problem \eqref{eqn:1+1+1-body} is symmetric with respect to the real axis, and the action functional $\mathcal{A}$ is invariant under the complex conjugation. Therefore, it is sufficient to prove this theorem on $[-T,0]$. Moreover, by Lemma \ref{lem:isolated}, there exists a moment $t_0\in(0,T]$, such that $z(t)$ is smooth on $(-t_0,0)$, $z(-t_0)\in\mathbb{C}^+_*$ and $z(0)=0$.

To prove (a). By Theorem \ref{thm:two-body}, since $z(t)$ is a minimizer, there exists no two-body collision on $[-T,0)$. Then by taking $t_0=T$, we see that (a) is a direct conclusion of Lemma \ref{lem:monotonity=0-} and Lemma~\ref{thm:main 3.2}.

To prove (b). The first two estimates of \eqref{eqn:sundman1} follow from Lemma \ref{lem:a*=alpha2} - \ref{lem:b*=alpha123}. To show the third estimate. Let $c(t)=\ddot r(t)/\frac{d^2}{dt^2}|q(t)|$. According to (\ref{eqn:ddotrh}) and \eqref{eqn:newton}, we have
\begin{align}\label{eqn:ddotlimit9}
c(t)-a(t)
& = -\frac{1}{a^2(t)}\left(h_{m/\mu}(a(t),\theta(t))+\frac{1}{\mu}\tan^2(\eta(t))\dot r^2(t)r(t)\right).
\end{align}
Combing Lemma~\ref{lem:a*=alpha2}, \ref{lem:a*=alpha13}, Lemma \ref{lem: zeros of h(a,theta)} and Lemma \ref{lem:dot theta d-}, we have $a_*=\lim_{t \rightarrow 0^-} a(t)\in\{\alpha_i\}_{i=1}^3$,
$$
\lim_{t \rightarrow 0^-}h_{m/\mu}(a(t),\theta(t))=h_{m/\mu}(\alpha_*,0)=0\quad \mathrm{and}\quad \lim_{t \rightarrow 0^-}\tan^2\eta(t)\frac{\dot{r}^2(t)}{a(t)}|q(t)|=0.
$$
Then as $t\rightarrow 0^-$, we obtain from \eqref{eqn:ddotlimit9} that $\lim_{t\rightarrow 0^-}c(t)=a_*\in\{\alpha_i\}_{i=1}^3$. In particular, $(b_1)$ and $(b_2)$ follow. The proof is now completed.
\end{proof}

\section{Application: Classical solutions with prescribed boundary angles}\label{sec4}

In this section, as an application of Theorem \ref{main1}, i.e. the Sundman-Sperling estimates to the restricted one-center-two-body problem \eqref{eqn:1+1+1-body}, we aim to prove Theorem \ref{thm:main1.1}, i.e. the existence of the collision-free solutions $z(t)$ jointing from the ray $e^{\phi i}\mathbb{R}^+$ to $e^{\phi_0i}\mathbb{R}^+$ in $t\in[-T,0]$, where $(\phi,\phi_0)\in[0,\pi]\times[0,\pi]$ with $\phi\neq \phi_0$. Similar in $[0,T]$ with $\phi,\phi_0$ switched. 
The strategy of this proof is first to apply Theorem~\ref{thm:two-body} to exclude the two-body collisions, then apply Theorem~\ref{main1} and the local deformation method to exclude the three-body collision for prescribed boundary angles $\phi,\phi_0$.


Recall from \eqref{eqn:action0} that the action functional
\begin{align*}
\mathcal A_{a,b}(z)=\int_{a}^{b}\frac{1}{2}|\dot z|^2+U(z,t)dt,\quad \forall z\in H^1([a,b],\mathbb{C}),
\end{align*}
where
the potential $U(z,t)$ is given in (\ref{eqn:potential}). It is well-known that for any $a<b$, $\mathcal A_{a,b}$ is a weakly lower semi-continuous on $H^1([a,b], \mathbb C)$ and the equation (\ref{eqn:1+1+1-body}) is the Euler-Lagrange equation of $\mathcal A_{a,b}$. This implies that the critical points of $\mathcal A_{a,b}$ in $H^1([a,b],\mathbb C)$ are weak solutions of the restricted one-center-two-body problem (\ref{eqn:1+1+1-body}) on $t \in [a,b]$.

Consider the path space
\begin{align*}
\Omega^{a,b}_{\phi_1,\phi_2}=\{x = re^{\theta i} \in H^1([a,b],\mathbb C): \ r\in \mathbb R^+,\ \theta(a)=\phi_1,  \  \theta(b)=\phi_2\}.
\end{align*}
We can obtain the following lemma.
\begin{Lem}\label{lem:minimizer}
Given $T>0$  and a collision Kepler system  $(q,c)$ which satisfies (\ref{eqn:newton}) and  $(Q1) - (Q3)$.
For any $\phi,  \phi_0 \in [0,\pi)$ with $\phi \neq \phi_0$,  the action functional $\mathcal A_{-T,0}$ (resp.   $\mathcal A_{0,T}$) attains its infimum on $\Omega^{-T,0}_{\phi,\phi_0}$ (resp.   $\Omega^{0,T}_{\phi_0,\phi}$).
\end{Lem}

\begin{proof}
Here, because of the reversibility of \eqref{eqn:1+1+1-body}, we only show the existence of minimizer of $\mathcal A_{-T,0}$ on $\Omega^{-T,0}_{\phi,\phi_0}$ and omit the case for $\mathcal A_{0,T}$ on $\Omega^{0,T}_{\phi_0,\phi}$.

Since $\mathcal{A}_{-T,0}$ is weakly lower semi-continuous on $H^1([a,b],\mathbb{C})$, it is sufficient to show $\mathcal A_{-T,0}$ is coercive in $\Omega^{-T,0}_{\phi,\phi_0}$,
that is $\mathcal A_{-T,0}(z)  \rightarrow +\infty \
\text{as} \ \|z\|_{L^2([-T,0],\mathbb C)} \rightarrow  +\infty$.

Let $z=re^{\theta i}\in \Omega^{-T,0}_{\phi,\phi_0}$.  Since $\theta(-T) = \phi$ and $\theta(0)=\phi_0$,  we have
\begin{align*}
\lvert \sin(\phi-\theta(t)) \rvert  \cdot \lvert z(t)\rvert \leq \lvert z(t)-z(-T)\rvert \ \ \text{and} \ \
 \lvert \sin(\phi_0-\theta(t))\rvert\cdot \lvert z(t)\rvert  \leq \lvert z(t)-z(0)\rvert,
\end{align*}
for $t \in [-T,0]$.

By Holder's inequality, we further obtain that
\begin{align*}
&  \lvert \sin(\phi-\theta(t))\rvert^2  \lvert z(t)\rvert^2 \leq \lvert z(t)-z(-T)\rvert^2
\leq \left( \int_{-T}^t \lvert \dot{z}(s)\rvert ds  \right)^2
\leq (t+T)\int_{-T}^t \lvert \dot{z}(s)\rvert^2 ds, \\
&  \lvert \sin(\phi_0-\theta(t))\rvert^2  \lvert z(t)\rvert^2 \leq \lvert z(t)-z(0)\rvert^2
\leq \left( \int_{t}^0 \lvert \dot{z}(s)\rvert ds  \right)^2
\leq (-t)\int_{t}^0 \lvert \dot{z}(s)\rvert^2 ds.
\end{align*}
for $t \in [-T,0]$. Then we compute that
\begin{align*}
\|\dot{z}\|^2_{L^2([-T,0],\mathbb C)}
&\geq \left( \frac{\lvert \sin(\phi-\theta(t))\rvert^2}{t+T} +\frac{\lvert \sin(\phi_0-\theta(t))\rvert^2}{-t} \right) \lvert z(t)\rvert^2 \\
&\geq \frac{1}{T} \left(\lvert \sin(\phi-\theta(t))\rvert^2+\lvert \sin(\phi_0-\theta(t))\rvert^2 \right) \lvert z(t)\rvert^2 \\
&\geq \frac{2}{T} C_{\phi,\phi_0} \lvert z(t)\rvert^2,
\end{align*}
where
\begin{align*}
C_{\phi,\phi_0}:=\min\left\{ \sin^2\left(\frac{\phi-\phi_0}{2} \right),  \cos^2\left(\frac{\phi-\phi_0}{2} \right) \right\}>0, \quad \forall \phi,\phi_0\in[0,\pi),
\end{align*}
and the last inequality can be checked by using direct computation. Therefore, we have
\begin{align*}
\|z\|_{L^2([-T,0],\mathbb C)} \leq \frac{T}{\sqrt{2C_{\phi,\phi_0}}} \|\dot{z}\|_{L^2([-T.0],\mathbb C)}.
\end{align*}
Since $\mathcal A_{-T,0}(z) \geq \frac{1}{2}\|\dot{z}\|^2_{L^2([-T,0])}$,  we conclude that
\begin{align*}
\mathcal A_{-T,0}(z)  \rightarrow +\infty \quad
\text{as} \quad \|z\|_{L^2([-T,0],\mathbb C)} \rightarrow  +\infty.
\end{align*}
This completes the proof.
\end{proof}

As previously stated, the minimizer obtained in Lemma~\ref{lem:minimizer} is a weak solution of the restricted one-center-two-body problem (\ref{eqn:1+1+1-body}). It becomes a classical solution if it does not encounter any collisions.
Therefore, to prove Theorem \ref{thm:main1.1}, 
it is sufficient to exclude the two-body and three-body collisions in the minimizer of $\mathcal A_{-T,0}$ on $\Omega_{\phi,\phi_0}^{-T,0}$. The proof on $t\in[0,T]$ is similar.

We first exclude the two-body collisions in the following theorem.
\begin{Thm}\label{thm:two-body1}
Given $T>0$ and a collision Kepler system $(q,c)$ satisfying (\ref{eqn:newton}) and  $(Q1) - (Q3)$. Let $\mathcal{A}_{-T,0}\ (\mathcal{A}_{0,T})$ be an action functional on $\Omega^{-T,0}_{\phi,\phi_0}\ (\Omega^{0,T}_{\phi,\phi_0})$ as in \eqref{eqn:action0} and assume $z(t)$ is an associated minimizer. Then $z(t)$ possesses no two-body collision on $[-T,0]\ (\text{or}\ [0,T])$. 
\end{Thm}
\begin{proof}
We only prove this theorem for $\mathcal A_{-T,0}$ on $\Omega_{\phi,\phi_0}^{-T,0}$. By applying Theorem \ref{thm:two-body}, it is sufficient to consider the case $\tau=-T$.

Firstly, without loss of generality, we can assume $z(t)\in \mathbb{C}^+_*$ is smooth on $(-T,0)$ by using conjugation and Theorem \ref{thm:two-body}. 
By assumption that $\phi\in[0,\pi)$, the two-body collision at $t=-T$ must be between $z$ and $c$. When the situation $\phi \neq 0$ or the situation $\phi =0$ with $\theta_{c,-T}^+\neq \pi$ occurs, we have $ \lvert \theta(-T)-\theta_{c,-T}^+ \rvert <\pi$. Then by using Proposition \ref{pro:local2}, we get a contradiction to the assume that $z$ is a minimizer.

When the situation $\phi = 0$ with $\theta_{c,-T}^+=\pi$ occurs,  the particle $z$ will lie on the negative real axis until next collision happens (at $\tau' \in (-T,0]$).  By Theorem \ref{thm:two-body}, we only need to consider the case $\tau'=0$.  In this situation, $z(-T)=z(0)=0$ and $z(t) \in (q(t),0) \subset (-\infty,0)$ on $(-T,0)$.  It is clear that $-z \in \Omega^{-T,0}_{\phi,\phi_0}$ and  $\mathcal A_{-T,0}(-z)<\mathcal A_{-T,0}(z)$ since $-z$ has same kinetic energy with $z$ and less potential energy than $z$.  This causes a contradiction and the proof now is completed.
\end{proof}

Now we are ready to prove Theorem \ref{thm:main1.1}.

\begin{proof}[Proof of Theorem \ref{thm:main1.1}]
According to Lemma \ref{lem:minimizer} and Theorem \ref{thm:two-body1}, we know that $\mathcal{A}_{-T,0}$ admits a minimizer $z(t)$ on $\Omega^{-T,0}_{\phi,\phi_0}$ with no two-body collision on $[-T,0]$. By the properties of minimizer, it is sufficient to prove (a), (c) and (d).

As in the proof of Theorem \ref{thm:two-body1}, we can assume $z(t)\in\mathbb{C}^+_*$ is smooth for all $t\in(-T,0)$.

To show (a). It is sufficient to exclude the three-body collision of $z$, i.e. $z(0)\neq 0$. Assume $z(0)=0$ by contradiction. Since $\phi_0,\phi\neq \pi$, by Theorem \ref{main1} (c.f. Theorem  \ref{thm:main 3.2}), we have $\theta_*^-=0$. Consider the following two cases.

Case 1: $\phi_0=0$. Write $z(t)=x(t)+iy(t)$. Choose an $\epsilon>0$ sufficiently small. We define a new path $\hat{z}_{\epsilon}(t)=\hat{x}_{\epsilon}(t)+i\hat{y}_{\epsilon}(t)$ by
\begin{eqnarray*}
  \hat{x}_{\epsilon}(t) :=
  \left\{
  \begin{array}{ll}
 \ x(t),  & \text{if $t \in [-T, -\epsilon]$},
  \vspace{1ex}\\
 x(-\epsilon),  & \text{if $t \in [-\epsilon,  0]$},
  \vspace{1ex}\\
  \end{array}
  \right.
 \quad\quad \hat{y}_{\epsilon}(t)=y(t), \quad \text{ for all $t \in [-T,0]$},
\end{eqnarray*}
see the left picture in Figure~\ref{fig:testpath}. It is easy to check that $\hat{z}_{\epsilon} \in \Omega^{-T,0}_{\phi,0}$. Since $\theta_*^-=0$, for any $\epsilon>0$ sufficiently small, we have $0<x(t) <x(-\epsilon)$ for every $t\in(-\epsilon,0)$. Then for any $t\in(-\epsilon,0)$, we conclude that
\begin{align*}
 \lvert \dot{z}(t) \rvert ^2- \lvert \dot{\hat{z}}_{\epsilon}(t) \rvert ^2
&=\dot{x}^2(t) \geq 0,
\end{align*}
\begin{align*}
\frac{\mu}{ \lvert z(t) \rvert }-\frac{\mu}{ \lvert \hat{z}_{\epsilon}(t) \rvert }
=\frac{\mu}{\sqrt{x^2(t)+y^2(t)}}-\frac{\mu}{\sqrt{x^2(-\epsilon)+y^2(t)}}>0,
\end{align*}
and
\begin{align*}
\frac{m}{ \lvert z(t)-q(t) \rvert }-\frac{m}{ \lvert \hat{z}_{\epsilon}(t)-q(t) \rvert }
=\frac{\mu}{\sqrt{(x(t)-q(t))^2+y^2(t)}}-\frac{\mu}{\sqrt{(x(-\epsilon)-q(t))^2+y^2(t)}}>0.
\end{align*}
These inequalities show that
$\mathcal A_{-T,0}(z)-\mathcal A_{-T,0}(\hat{z}_{\epsilon})>0$.
This leads to a contradiction to the minimizer. Case 1 is proved.

Case 2: $\phi_0 \in (0,\pi/2]$. Choose an $\epsilon>0$ sufficiently small and define $\tilde{z}_{\epsilon}(t)=\tilde{x}_{\epsilon}(t)+i\tilde{y}_{\epsilon}(t)$ by
\begin{eqnarray*}
  \tilde{x}_{\epsilon}(t) :=
  \left\{
  \begin{array}{ll}
 x(t),  & \text{if $t \in [-T, -\epsilon]$},
  \vspace{1ex}\\
   x(t),  & \text{if $t \in (-\epsilon, t_{\epsilon}]$},
  \vspace{1ex}\\
 x(t_{\epsilon}),  & \text{if $t \in (t_{\epsilon},  0]$},
  \vspace{1ex}\\
  \end{array}
  \right.
 \quad\quad \tilde{y}_{\epsilon}(t):=
 \left\{
 \begin{array}{ll}
 \ y(t),  & \text{if $t \in [-T,-\epsilon]$},
 \vspace{1ex}\\
 y(-\epsilon), & \text{if $t \in (-\epsilon,t_{\epsilon}]$},
 \vspace{1ex}\\
 y(-\epsilon),  & \text{if $t \in (t_{\epsilon},  0]$},
  \vspace{1ex}\\
 \end{array}
 \right.
\end{eqnarray*}
see the right picture in Figure~\ref{fig:testpath}. where $t_{\epsilon} \in (-\epsilon,0]$ is a moment with $x(t_{\epsilon})+iy(-\epsilon) \in e^{\phi_0i}\mathbb R^+$.
Since $\theta_*^-=0$, for $\epsilon>0$ sufficiently small, we have $\theta(-\epsilon)< \phi_0$, then the moment $t_{\epsilon}$ exists. We see that $\tilde{z}_{\epsilon} \in \Omega^{-T,0}_{\phi,\phi_0}$. Moreover, for $\epsilon>0$ sufficiently small, we have
\begin{align*}
0<x(t) \leq \tilde{x}_{\epsilon}(t)\quad \text{and} \quad  0<y(t) <\tilde{y}_{\epsilon}(t), &  \quad \text{ for all } t\in(-\epsilon,0).
\end{align*}
Then for $t\in(-\epsilon,0)$, we conclude that
\begin{align*}
 \lvert \dot{z}(t) \rvert ^2- \lvert \dot{\tilde{z}}_{\epsilon}(t) \rvert ^2
&\geq \dot{y}^2(t) \geq 0,
\end{align*}
\begin{align*}
\frac{\mu}{ \lvert z(t) \rvert }-\frac{\mu}{ \lvert \tilde{z}_{\epsilon}(t) \rvert }
=\frac{\mu}{\sqrt{x^2(t)+y^2(t)}}-\frac{\mu}{\sqrt{\tilde{x}_{\epsilon}^2(t)+\tilde{y}_{\epsilon}^2(t)}}>0,
\end{align*}
and
\begin{align*}
\frac{m}{ \lvert z(t)-q(t) \rvert }-\frac{m}{ \lvert \tilde{z}_{\epsilon}(t)-q(t) \rvert }
=\frac{\mu}{\sqrt{(x(t)-q(t))^2+y^2(t)}}-\frac{\mu}{\sqrt{(\tilde{x}_{\epsilon}(t)-q(t))^2+\tilde{y}_{\epsilon}^2(t)}}>0.
\end{align*}
These inequalities implies that
$\mathcal A_{-T,0}(z)-\mathcal A_{-T,0}(\tilde{z}_{\epsilon})>0$.
This leads to a contradiction to the minimizer. Case 2 is proved and the proof of (a) is completed.

To show (c). Recall that $z(t)$ is smooth on $(-T,0)$. By Corollary \ref{cor:monotonity1}, if (c) is false, then either $z(t)\in \mathbb{R}^+$ or $z(t)\in \mathbb{R}^-$. However, since $z(-T)\neq 0\neq z(0)$ by (a), $\phi\in[0,\pi)$ and $\phi\neq \phi_0$, neither of them happens. This gives a contradiction. Moreover, if $\min\{\phi,\phi_0\}=0$, then $\theta(t)$ must be strictly monotone on $[-T,0]$, i.e. $t_*\in\{0,\pi\}$. Otherwise, if $t_*\in(-T,0)$, then $\theta(t_*)=0$, which contradicts to the strictness of the monotonicity for $\theta(t)$ on $[-T,t_*]$ and $[t_*,0]$. Hence, (c) holds.

Finally, (d) follows directly from the variational computation of $\mathcal{A}_{-T,0}$ on $\Omega^{-T,0}_{\phi,\phi_0}$ at the collision-free minimizer $z|_{[-T,0]}$.
\end{proof}
\begin{Rmk}
When $\phi_0\in(\pi/2,\pi)$, the three-body collision is difficult to exclude. In fact, unlike the case of two-body collisions, the three-body collision involving two different singularities, which are asymptotic to each other as $t\rightarrow 0^\pm$. The behavior of these two singularities highly impact the action of the local deformation pathes. This causes a huge difficulty to the exclusion of the three-body collision.
\end{Rmk}

\begin{figure}[ht]
\begin{center}
\centering
\includegraphics[width=1\textwidth]{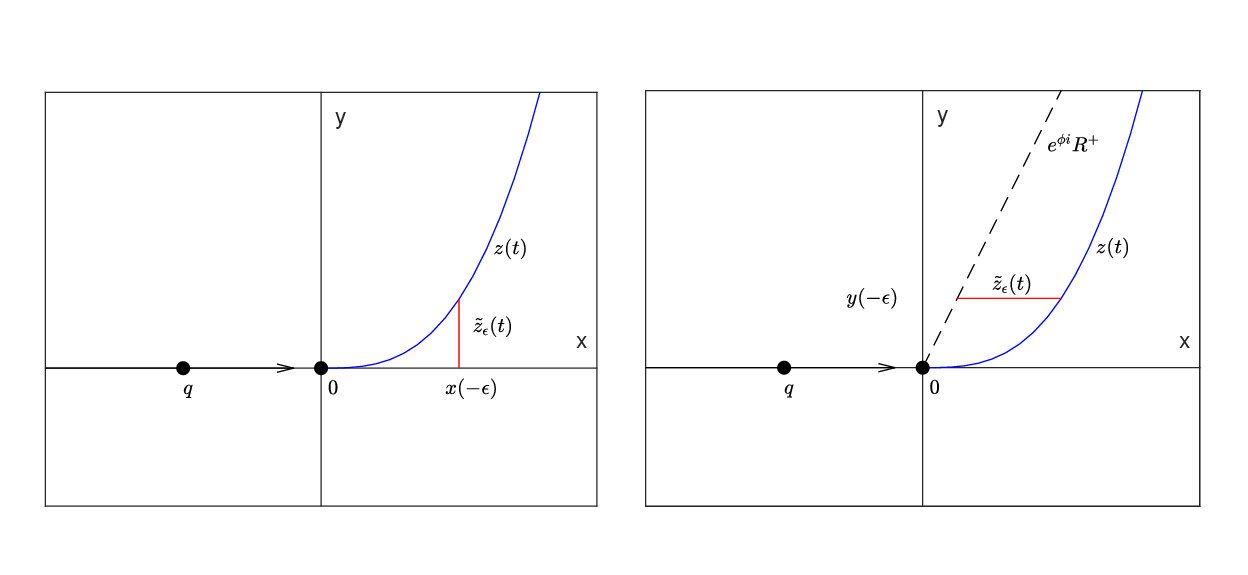}
\caption{Local deformation paths.}
\label{fig:testpath}
\end{center}
\end{figure}

\section{Appendix: Periodic and quasi-periodic solutions}\label{sec:periodic}

In one of the pioneer work \cite{Gordon77} of variational method to the $N$-body problem, Gordon mentioned a class of extended Kepler collision solution with negative energy which allow the bodies bounce towards arbitrary directions after collisions.

Let $\psi \in (-\pi,\pi)$ and  $c \equiv 0$.
In this appendix, we consider a special class of extended Kepler collision solution $q_{\psi}$ of (\ref{eqn:newton}) which has negative energy and reflect a fixed angle $\psi$ after each collision. Notice that $q_{\psi}$ is periodic ($\psi/\pi \in \mathbb Q$) or quasi-periodic ($\psi / \pi \notin \mathbb Q$) since the energy of $q_{\psi}$ is constant except the moment of collision \cite[Sec.3.3]{Chenciner02}.

More precisely, without loss of generality, we assume the extended Kepler collision orbit $q_{\psi}$ satisfies the following conditions:
\begin{enumerate}
\item[$(Q1_{\psi})$]  $q_{\psi}$ collides with $c$ at moment $t=2kT$ for each $k \in \mathbb Z$.
\item[$(Q2_{\psi})$] $q_{\psi}$ is smooth on $(2kT,2(k+1)T)$,  that is $q_{\psi} \neq 0$ on $(2kT,2(k+1)T)$,  for each $k \in \mathbb Z$.
\item[$(Q3_{\psi})$] $q_{\psi}|_{(2kT,2(k+1)T)}$ lies on the ray $e^{(\pi-k\psi)i}\mathbb R^+$,  for each  $k \in \mathbb Z$.
\end{enumerate}

According to assumptions $(Q1_{\psi})-(Q3_{\psi})$, we obtain the following properties:
\begin{enumerate}
\item[$(Q4_{\psi})$] $q_{\psi}((2k+1)T-t)=q_{\psi}((2k+1)T+t)$ on $[0,T]$,  for each $k \in \mathbb Z$.
\item[$(Q5_{\psi})$] $\dot{q}_{\psi}((2k+1)T)=0$,  for each $k \in \mathbb Z$.
\end{enumerate}

\begin{figure}[ht]
\begin{center}
\includegraphics[width=0.75\textwidth]{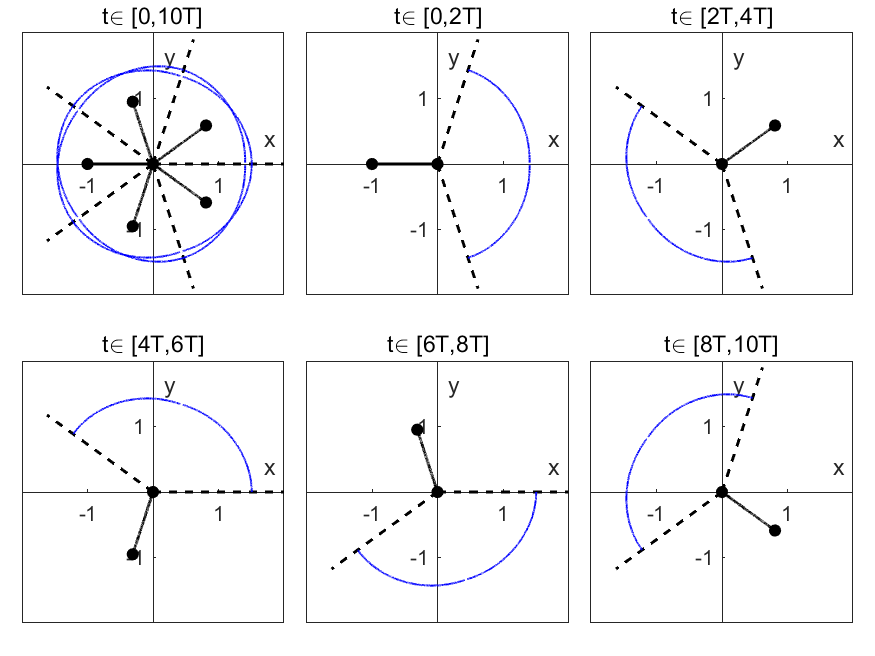}
\caption{Extended Kepler collision orbit $q_{-4\pi/5}$ and the solution $z_{-4\pi/5}$.
}
\label{fig:seprate}
\end{center}
\end{figure}

Consider the extended Kepler collision system $(q_{\psi},c)$ which satisfies $(Q1_{\psi})-(Q3_{\psi})$. According to Theroem~\ref{thm:main1.1} and  ~\ref{thm:main1.2}, there exists two collision-free solutions $z_{\psi}|_{[0,T]}\in\Omega^{0,T}_{\psi/2,0}$ and $z_{\psi}|_{[T,2T]}\in\Omega^{T,2T}_{0,-\psi/2}$ of the restricted one-center-two-body problem (\ref{eqn:1+1+1-body}) satisfying
(see Figure~\ref{fig:seprate})
\begin{enumerate}
\item[$\bullet$] $ z_{\psi}(t) = \bar z_{\psi}(2T-t)$ on $[0,T]$ and $z_{\psi}$ is smooth at $t=T$.
\item[$\bullet$] $ z_{\psi}(0) \in e^{\psi i/2}\mathbb R^+, z_{\psi}(T) \in \mathbb R^+$ and $z_{\psi}(2T) \in e^{-\psi i/2}\mathbb R^+$.
\item[$\bullet$] $z_{\psi}$ is orthogonal to $ e^{\psi i/2}\mathbb R^+$, $\mathbb R^+$ and $e^{-\psi i/2}\mathbb R^+$ at moment $t=0,T,2T$, respectively.
\end{enumerate}

Moreover,  by the symmetry of $q_\phi$,
the domain of solution $z_{\psi}|_{[0,2T]}$ can be extended to
$\mathbb R$  by choosing $z_{\psi}$ as following (see Figure~\ref{fig:seprate})
\begin{align*}
z_{\psi}(t)=z_{\psi}(\hat{t})e^{-k\psi i}, \quad \text{ where } k \in \mathbb Z, \ \hat{t}\in[0,2T)\ \text{with}\ t=2kT+\hat{t}.
\end{align*}
Note that $z_{\psi}(t)$ is well-defined on $\mathbb{R}$. For each $k\in \mathbb{Z}$, $z_{\psi}$ is smooth and orthogonal to $e^{\left(1/2-k\right)\psi i}\mathbb R^+$ at $t=2kT$. In particular, the three bodies $(z_{\psi}, q_{\psi},c)$ form a periodic ($\psi / \pi \in \mathbb Q$) or quasi-periodic ($\psi / \pi \notin \mathbb Q$) of the restricted one-center-two-body system.

\begin{Dataava}
All data needed to evaluate the conclusions in the paper are present in the paper. Additional data related to this paper may be requested from the authors.    \vspace{1em}
\end{Dataava}

\begin{Conflict}
The authors declare that they have no conflict of interests.   \vspace{0.5em}
\end{Conflict}

\begin{Acknow}
It is a pleasure to thank K.C. Chen and W.T. Kuang for discussions.
Hsu is supported by National Natural Science Foundation of China under grant (12101363,  12271300),
Natural Science Foundation of Shandong Province, China under grant (ZR2020QA013),
National Science Foundation for Young Scientists of Fujian Province under grant (2023J01123)
and Scientific Research Funds of Huaqiao University under grant (22BS101).
Liu is supported by National Natural Science Foundation of China under grants (12071255, 12171281).
\end{Acknow}

\end{document}